\documentclass{amsart}

\usepackage{amssymb}
\usepackage[hidelinks]{hyperref}
\usepackage{tikz}
\usetikzlibrary{cd}
\usepackage{float}

\theoremstyle{plain}
\newtheorem{theorem}{Theorem}[section]
\newtheorem{lemma}[theorem]{Lemma}
\newtheorem{corollary}[theorem]{Corollary}
\newtheorem{proposition}[theorem]{Proposition}

\theoremstyle{definition}
\newtheorem{definition}[theorem]{Definition}

\newtheorem{example}[theorem]{Example}

\newtheorem{question}{Question}

\theoremstyle{remark}
\newtheorem*{remark}{Remark}

\newcommand{\cA}{\mathcal{A}}

\newcommand{\cC}{\mathcal{C}}

\newcommand{\cD}{\mathcal{D}}
\newcommand{\cE}{\mathcal{E}}
\newcommand{\cF}{\mathcal{F}}
\newcommand{\cG}{\mathcal{G}}

\newcommand{\cI}{\mathcal{I}}
\newcommand{\I}{\cI}
\newcommand{\cJ}{\mathcal{J}}
\newcommand{\J}{\cJ}
\newcommand{\cK}{\mathcal{K}}
\newcommand{\K}{\cK}

\newcommand{\cN}{\mathcal{N}}
\newcommand{\cP}{\mathcal{P}}

\newcommand{\cU}{\mathcal{U}}

\newcommand{\cZ}{\mathcal{Z}}
\newcommand{\continuum}{\mathfrak{c}}
\newcommand{\cc}{\continuum}
\newcommand{\bnumber}{\mathfrak{b}}
\newcommand{\bb}{\bnumber}
\newcommand{\dnumber}{\mathfrak{d}}

\newcommand{\dd}{\dnumber}

\DeclareMathOperator{\add}{add}

\DeclareMathOperator{\adds}{\add^*}

\newcommand{\fin}{\mathrm{Fin}}
\newcommand{\finShort}{\mathrm{F}}

\newcommand{\ED}{\mathcal{ED}}

\DeclareMathOperator{\dom}{dom}
\DeclareMathOperator{\ran}{ran}
\DeclareMathOperator{\cf}{cf}

\begin{document}


\title{Yet another ideal version of the bounding number}


\author{Rafa\l{} Filip\'{o}w}
\address[Rafa\l{}~Filip\'{o}w]{Institute of Mathematics\\ Faculty of Mathematics, Physics and Informatics\\ University of Gda\'{n}sk\\ ul.~Wita Stwosza 57\\ 80-308 Gda\'{n}sk\\ Poland}
\email{Rafal.Filipow@ug.edu.pl}
\urladdr{http://mat.ug.edu.pl/~rfilipow}

\author[Adam Kwela]{Adam Kwela}
\address[Adam Kwela]{Institute of Mathematics\\ Faculty of Mathematics\\ Physics and Informatics\\ University of Gda\'{n}sk\\ ul.~Wita  Stwosza 57\\ 80-308 Gda\'{n}sk\\ Poland}
\email{Adam.Kwela@ug.edu.pl}
\urladdr{http://kwela.strony.ug.edu.pl/}


\date{\today}






\begin{abstract}
Let $\I$ be an ideal on $\omega$. For $f,g\in\omega^\omega$ we write $f \leq_\I  g$ if $f(n)  \leq  g(n)$ for all $n\in\omega\setminus A$ with some $A\in\I$. Moreover, we denote $\mathcal{D}_\I=\{f\in\omega^\omega: f^{-1}[\{n\}]\in\I \text{ for every $n\in \omega$}\}$ (in particular, $\mathcal{D}_\fin$ denotes the family of all finite-to-one functions).

We examine cardinal numbers $\bb(\geq_\I\cap (\cD_\I \times \cD_\I))$ and  $\bb(\geq_\I\cap (\cD_\fin\times \cD_\fin))$ describing the smallest sizes of unbounded from below with respect to the order $\leq_\I$ sets in
$\mathcal{D}_\fin$ and $\mathcal{D}_\I$, respectively. 
For a maximal ideal $\I$, these cardinals were investigated by M.~Canjar in connection with coinitial and cofinal subsets of the ultrapowers.

We show that  $\bb(\geq_\I\cap (\cD_\fin \times \cD_\fin)) =\bb$ for all ideals $\I$ with the Baire property and that $\aleph_1 \leq \bb(\geq_\I\cap (\cD_\I \times \cD_\I)) \leq\bb$ for all coanalytic weak P-ideals (this class contains all $\bf{\Pi^0_4}$ ideals). What is more, we give examples of Borel (even $\bf{\Sigma^0_2}$) ideals $\I$ with $\bb(\geq_\I\cap (\cD_\I \times \cD_\I))=\bb$ as well as with 
 $\bb(\geq_\I\cap (\cD_\I \times \cD_\I)) =\aleph_1$.

We also study cardinals 
$\bb(\geq_\I\cap (\cD_\J \times \cD_\K))$
describing the smallest sizes of sets in 
$\mathcal{D}_\K$ not bounded from below with respect to the preorder $\leq_\I$ by any member of $\mathcal{D}_\J$.
Our research is 
partially motivated by the study of ideal-QN-spaces: those cardinals describe the smallest size of a space which is not ideal-QN.
\end{abstract}


\maketitle
\tableofcontents


\section{Introduction}

For an ideal $\I$ on $\omega$ we denote $\mathcal{D}_\I=\{f\in\omega^\omega: f^{-1}[\{n\}]\in\I \text{ for every $n\in \omega$}\}$ (in particular, $\mathcal{D}_\fin$ denotes the family of all finite-to-one functions) and write $f\leq_\I g$ if $\{n\in\omega:f(n)>g(n)\}\in\I$, where $f,g\in\omega^\omega$. 
By $\bb(\geq_\I\cap (\cD_\J \times \cD_\K))$ we denote 
the smallest sizes of sets in 
$\mathcal{D}_\K$ not bounded from below with respect to the order $\leq_\I$ by any member of $\mathcal{D}_\J$. 
(We restrict  ourself to functions from $\cD_{\J}$  instead of all functions from $\omega^\omega$, because every subset of $\omega^\omega$ is bounded from below with respect to $\leq_\I$ by the constant zero function.)

In 1980s, Canjar~\cite{Canjar2,Canjar,CanjarPhD}
studied the smallest sizes of cofinal and coinitial subsets in ultrapowers $\omega^\omega/\cU$ ordered by $\leq_{\I}$ for ultrafilters $\cU=\I^*$. 
Among others, Conjar proved that 
consistently 
$\bb(\geq_\I\cap (\cD_\fin\times \cD_\fin))$
and 
$\bb(\geq_\I\cap (\cD_\I \times \cD_\I))$ can be equal to any regular cardinal between $\aleph_1$ and $\bnumber$ for some maximal ideal $\I$.

	In this paper, we examine cardinals 
$\bnumber(\geq_\I\cap (\cD_\K\times\cD_\J))$
for various triples of ideals $(\I,\J,\K)$, but we do not require from ideals $\I,\J$ and $\K$ to be maximal. 
We obtained some facts for arbitrary ideals $\I, \J$ and $\K$, however, the most interesting results we proved for Borel ideals.
For instance, we showed that there are Borel ideals $\I$ with 
$\bb(\geq_\I\cap (\cD_\I \times \cD_\I))=\aleph_1$ and 
$\bb(\geq_\I\cap (\cD_\fin \times \cD_\fin))=\bnumber$ provable in ZFC.

To make life easier, we restrict our study only to  eight cases with at most two different ideals among $\I,\J,\K$ and with at least one of them equal to the ideal $\fin$.
It will follow from results of Section~\ref{sec:basic-relations-between-the-bounding-numbers}
that   in fact  we can restrict our study only to four 
cases: $\bnumber(\geq_\I\cap (\cD_\I\times\cD_\I))$, $\bnumber(\geq_\I\cap (\cD_\fin\times\cD_\I))$, $\bnumber(\geq_\I\cap (\cD_\fin\times\cD_\fin))$ and $\bnumber(\geq_\fin\cap (\cD_\I\times\cD_\I))$.

The paper is organized in the following way.
In Section~\ref{sec:BoundingNumbers}
we show that $\bnumber(>_\I\cap (\cD_\K\times\cD_\J)) = \bb(\I,\J,\K)$ and 
$\bnumber(\geq_\I\cap (\cD_\K\times\cD_\J)) = \bb_s(\I,\J,\K)$, where $\bb(\I,\J,\K)$ and $\bb_s(\I,\J,\K)$
are cardinals considered by Filip\'{o}w and Stanszewski in \cite{MR3269495,MR3624786}. 
This provides us with a very useful combinatorial characterizations 
of the considered cardinals, which we use almost exclusively in the rest of the paper.

Section~\ref{sec:basic-relations-between-the-bounding-numbers} is devoted to showing the basic properties of considered $\bb$-like numbers. Also, we prove that $\bnumber(\geq_\fin\cap(\cD_\I\times\cD_\I))=\min\{\bb,\adds(\I)\}$ and present some diagrams summarizing basic relationships between considered $\bb$-like numbers

In Section~\ref{sec:Sums-and-products-of-ideals}, we calculate $\bb$-like numbers for direct sums and Fubini products of ideals. Moreover, we compute $\bb$-like numbers for not tall ideals. The results of that section are used in Section~\ref{sec:ConsistencyResults} to obtain examples of ideals with distinct values of various  $\bb$-like cardinals.

In Section~\ref{sec:NiceIdeals}, we examine $\bb$-like cardinals for nice ideals, where nice means ideals with the Baire property, $\omega$-diagonalizable ideals or definable ideals. We show that $\bnumber(\geq_\I\cap (\cD_\fin\times\cD_\fin)) =\bb$ for all ideals $\I$ with the Baire property and infer that for P-ideals with the Baire property all considered cardinals equal $\bb$. Moreover, we show that $\bnumber(\geq_\I\cap (\cD_\fin\times\cD_\I)) =\bb$ for coanalytic weak P-ideals. The latter gives an upper bound for $\bnumber(\geq_\I\cap (\cD_\I\times\cD_\I))$ in the case of coanalytic weak P-ideals (by a result of Debs and Saint
Raymond, this class contains all $\bf{\Pi^0_4}$ ideals).

Section~\ref{sec:bb-equals-aleph-one} is devoted to ideals $\I$ with $\bnumber(\geq_\I\cap (\cD_\I\times\cD_\I))=\aleph_1$. In particular, we give examples of such ideals among not tall $\bf{\Sigma^0_2}$ ideals and among tall $\bf{\Sigma^0_2}$ ideals.
 
In Section~\ref{sec:ConsistencyResults},
we show that consistently the values of considered cardinals can be pairwise distinct. 

In the literature, there are considered other ideal versions of the bounding number $\bnumber$. For instance, Farkas and Soukup~\cite{MR2537837} consider $\bnumber(\leq_\I\cap (\omega^\omega\times\omega^\omega)$ (Canjar in \cite{Canjar2} proved that consistently there are ideals $\I$ with $\bnumber(\leq_\I\cap (\omega^\omega\times\omega^\omega)\neq  \bnumber(\geq_\I\cap (\cD_\fin\times\cD_\fin))$; see also \cite{Canjar}), and 
Brendle and Mej\'{\i}a~\cite{MR3247032}
consider the number $\bnumber(\I)$ defined as the smallest $\kappa$ such that there is an $\I-(\omega,\kappa)$-gap in $(\cP(\omega),\subseteq^\I)$ and call it the Rothberger number of $\I$ (in general, $\bnumber(\I)\neq \bnumber(\geq_\I\cap (\cD_\I\times\cD_\I))$ because the former is not defined for maximal ideals whereas the latter is).


\section{Preliminaries}
\label{sec:preliminaries}

By $\omega$ we denote the set of all natural numbers.
We identify  a natural number $n$ with the set $\{0, 1,\dots , n-1\}$. (Thus, for instance, $n\setminus k$ means the set $\{i\in\omega: k\leq i<n\}$).
We write $A\subseteq^*B$ or $B\supseteq^* A$ if $A\setminus B$ is finite.
For a set $A$
and
a finite or infinite cardinal number  $\kappa$, we write $[A]^{\kappa} =\{B\subseteq A: |B|=\kappa\}$.


\subsection{Ideals}

An \emph{ideal on a set $X$} (in short \emph{ideal}) is a  family $\I\subseteq\cP(X)$ that satisfies the following properties:
\begin{enumerate}
\item if $A,B\in \I$ then $A\cup B\in\I$,
\item if $A\subseteq B$ and $B\in\I$ then $A\in\I$,
\item $\I$ contains all finite subsets of $X$,
\item $X\notin\I$.
\end{enumerate}
For an ideal $\I$, 
we write $\I^+=\{A\subseteq X: A\notin\I\}$ and call it the \emph{coideal of $\I$}, 
and we write $\I^*=\{A\subseteq X: X\setminus A\in\I\}$ and call it the \emph{dual filter of $\I$}.

The ideal of all finite subsets of an infinite set $X$
is denoted by $\fin(X)$ (or $\fin$ for short).

Let $\cA\subseteq\cP(X)$. If $X$ cannot be covered by finitely many members of $\cA$ then the smallest ideal containing $\cA$ i.e.~the ideal $\I=\{B\subset X:  \exists A_1,\dots,A_n \in \cA\, (B\subseteq^*  A_1\cup\dots\cup A_n)\}$ is called \emph{the ideal generated by $\cA$}.

An ideal $\I$ on $X$ is \emph{P-ideal} (\emph{weak P-ideal}, resp.) if for any countable family $\cA\subseteq\I$ there is $B\in \I^*$ ($B\in \I^{+}$, resp.) such that $A\cap B\in \fin(X)$ for every $A\in \cA$. 

While P-ideals can be considered as a classical notion, weak P-ideals are gaining in popularity nowadays. Usefulness of weak P-ideals follows from the fact that this property can be characterized in various manners. For instance, an  ideal $\I$ is a weak P-ideal if and only if one of the following conditions hold:
\begin{itemize}
	\item $\I$ is not above the ideal $\fin\otimes\fin$ in the Kat\v{e}tov order (see e.g.~\cite{MR3423409} where also other orders are used for similar characterizations);
	
	\item $\I$ and $\I^*$ can be separated by an $F_\sigma$ set (\cite{MR2491780});
	
	\item Player I does not have a winning strategy in a game  introduce by Laflamme in \cite{Laflamme} (see Subsection~\ref{sec:coanalytic-ideals} of the present paper for a definition of this game); 

	\item  $\I$ is $\omega$-diagonalizable by $\I^*$-universal sets (see Subsection~\ref{sec:omega-diag} for a definition of this notion).
\end{itemize}
The last characterization is  technical, but its combinatorial character turns out to be very useful for working with weak P-ideals: we use it for examining some $\bnumber$-like numbers in Subsection~\ref{sec:omega-diag} (for other applications of this characterization see e.g.~\cite{MR2899832,MR2990109,MR3624783,MR2491780}).

We say that ideals $\I,\J$ on $X$ are \emph{orthogonal} (in short: $\I\perp\J$)
if there is $A\in \J$ with $X\setminus A\in \I$

An ideal $\I$ on $X$ is \emph{tall} if for every infinite $A\subseteq X$ there is an infinite $B\in\I$ such that $B\subset A$.

For an ideal $\I$ on $X$ and for $A\notin\I$ we define 
$\I\restriction A = \{B\subseteq A: B\in \I\}$. It is easy to see that $\I\restriction A$  is an ideal on $A$.

The vertical section of a set   $A\subseteq X\times Y$ at a point $x\in X$ is defined 
by
$(A)_x = \{y\in Y : (x,y)\in A\}$.

For ideals $\I,\J$ on $X$ and $Y$ respectively we define the following new ideals:
\begin{enumerate}
	\item
	$\I\oplus\J = \{A\subseteq X\times\{0\}\cup Y\times \{1\}: \{x:(x,0)\in A\}\in \I\land \{y:(y,1)\in A\}\in \J\}\}$,
	\item $\I\oplus\cP(\omega) = \{A\subseteq X\times\{0\}\cup \omega\times \{1\}: \{n:(n,0)\in A\}\in \I\}$,
	\item
	$\I\otimes \J = \{A\subseteq X\times Y: \{x:(A)_x\notin\J\}\in \I\}$,
	\item
$\I\otimes \{\emptyset\} = \{A\subseteq X\times \omega: \{x:(A)_x\ne\emptyset\}\in \I\}$.
	\item
$\{\emptyset\} \otimes  \J= \{A\subseteq \omega\times Y : (A)_x\in \J\text{ for all $x$}\}$.
\end{enumerate}

For any $n\geq 1$ we define the ideals $\fin^n$  in the following manner: $\fin^1=\fin$, and $\fin^{n+1}=\fin\otimes\fin^n$.

By identifying sets of natural numbers with their characteristic functions,
we equip $\cP(\omega)$ with the topology of the Cantor space $\{0,1\}^\omega$ and therefore
we can assign topological complexity to ideals on $\omega$.
In particular, an ideal $\I$ is Borel (has the Baire property) if $\I$ is Borel (has the Baire property) as a subset of the Cantor space. 

	In the sequel, we use the convention that $\min \emptyset = \continuum^+$.

For an ideal $\I$ on $\omega$ we define 
$$\adds(\I)=\min(\{|\cF|: \cF\subseteq\I\land \forall A\in \I\,\exists F\in \cF\,(|F\setminus A|=\aleph_0)\}.$$
Note that $\adds(\I)=\aleph_0$ for every non P-ideal
and
$\add(\I)\geq \aleph_1$ for every P-ideal. 
Moreover, it is easy to see that $\add^*(\I)=\continuum^+$ if and only if $\I=\{C\subseteq\omega: |C\setminus A|<\aleph_0\}$ for some $A\subset\omega$. For instance, $\adds(\fin)=\continuum^+$.


\section{Various bounding numbers}
\label{sec:BoundingNumbers}


\subsection{The ordinary bounding number}

For $f,g\in \omega^\omega$ we write $f\leq^*g$ if $f(n)\leq g(n)$ for all but finitely many $n\in \omega$.
The \emph{bounding number} $\bnumber$ is the smallest size of $\leq^*$-unbounded subset of $\omega^\omega$  i.e.
$$\bnumber = \min\{|\cF|: \cF\subseteq\omega^\omega\land \neg(\exists g\in \omega^\omega\forall f\in \cF\, (f\leq^* g))\}.$$
The \emph{dominating number} $\dnumber$ is the smallest size of $\leq^*$-dominating subset of $\omega^\omega$  i.e.
$$\dnumber = \min\{|\cF|: \cF\subseteq\omega^\omega\land \forall g\in \omega^\omega\exists f\in \cF\, (g\leq^* f)\}.$$


\subsection{The bounding numbers \`{a} la Vojt\'{a}\v{s}}

If $R$ is a binary relation, then by $\dom(R)$ and $\ran(R)$ we denote the domain and range of $R$ respectively i.e.
$\dom(R)=\{x:\exists y\, (x,y)\in R\}$ and $\ran(R) = \{y: \exists x\, (x,y)\in R\}$.
A set $B\subseteq \dom(R)$ is called \emph{$R$-unbounded} if for every $y\in \ran(R)$ there is $x\in B$ with $(x,y)\notin R$.
A set $D\subseteq \ran(R)$ is called \emph{$R$-dominating} if for every $x\in \dom(R)$ there is $y\in D$ with $(x,y)\in R$.

	\begin{definition}[Vojt\'{a}\v{s}~\cite{MR1234291}]
Let $R$ be a binary relation.
\begin{equation*}
\begin{split}
	\bnumber(R) & = \min(\{|B|: \text{$B$ is $R$-unbounded set}\}.\\
\dnumber(R) & = \min(\{|D|: \text{$D$ is $R$-dominating set}\}.
\end{split}
\end{equation*}
\end{definition}

It is easy to see that 
the bounding  number $\bnumber$ is equal to the bounding number of the relation
$\leq^*\cap (\omega^\omega\times\omega^\omega)$
i.e.
$\bnumber = \bnumber(\leq^*\cap (\omega^\omega\times\omega^\omega))$, and similarly 
$\dnumber = \dnumber(\leq^*\cap (\omega^\omega\times\omega^\omega))$.


\subsection{The bounding numbers \`{a} la Canjar}

\begin{definition}
For an ideal $\I$ we define the relation $\leq_\I = \{(f,g)\in \omega^\omega\times \omega^\omega: \{n\in \omega: f(n)>g(n)\}\in \I\}$. We write $f\leq_\I g$ if $(f,g)\in \leq_\I$. 
In a similar manner we define $<_\I$, $\geq_\I$ and $>_\I$.
\end{definition}

\begin{definition}
For an ideal $\I$ we define
$$\cD_\I
=
\{f\in\omega^\omega: f^{-1}[\{n\}]\in\I \text{ for every $n\in \omega$}\}.$$
\end{definition}

In \cite{Canjar2,Canjar,CanjarPhD}, Canjar
studied cardinals $\bb(\geq_\I\cap (\cD_\I\times \cD_\I))$ and $\bb(\geq_\I\cap (\cD_\fin\times \cD_\fin))$ for some maximal ideals $\I$. In this paper we will examine the bounding numbers 
$\bnumber(>_\I\cap (\cD_\K\times\cD_\J))$ and 
$\bnumber(\geq_\I\cap (\cD_\K\times\cD_\J))$ for various triples of ideals $(\I,\J,\K)$.

\begin{proposition}
\label{prop:cofinality-of-b-number}
Let $\I,\J,\K$ be ideals on $\omega$.
\begin{enumerate}

\item 
$\bnumber(>_\I\cap (\cD_\K\times\cD_\J))\leq \bnumber(\geq_\I\cap (\cD_\K\times\cD_\J))$.\label{prop:cofinality-of-b-number-inequality}

\item $\bnumber(\geq_\I\cap (\cD_\J\times\cD_\J)) \leq \cf(\bnumber(\geq_\I\cap (\cD_\K\times\cD_\J)))$.\label{prop:cofinality-of-b-number-cofinality}

\item $\bnumber(>_\I\cap (\cD_\J\times\cD_\J)) \leq \cf(\bnumber(>_\I\cap (\cD_\K\times\cD_\J)))$.\label{prop:cofinality-of-b-number-cofinality2}

\item If $\J\subseteq\K$, then $\bnumber(\geq_\I\cap (\cD_\K\times\cD_\J))$ and $\bnumber(>_\I\cap (\cD_\K\times\cD_\J))$ are regular cardinals.\label{prop:cofinality-of-b-number-regular}

\end{enumerate}
\end{proposition}

\begin{proof}
(\ref{prop:cofinality-of-b-number-inequality}) Obvious.

(\ref{prop:cofinality-of-b-number-cofinality})
Let $\cF\subset \cD_\K$ be an $\geq_\I\cap (\cD_\K\times\cD_\J)$-unbounded family of cardinality 
$\bnumber(\geq_\I\cap (\cD_\K\times\cD_\J))$ 
and $\cF = \bigcup_{\alpha<\cf(\bnumber(\geq_\I\cap (\cD_\K\times\cD_\J))}\cF_\alpha$ with $|\cF_\alpha|<\bnumber(\geq_\I\cap (\cD_\K\times\cD_\J))$ for every $\alpha$.

For every $\alpha$ we find $g_\alpha\in \cD_\J$ with $f\geq_\I g_\alpha$ for every $f\in \cF_\alpha$.
Let $\cG=\{g_\alpha:\alpha<\cf(\bnumber(\geq_\I\cap (\cD_\K\times\cD_\J))\}$.

Suppose to the contrary that 
$\cf(\bnumber(\geq_\I\cap (\cD_\K\times\cD_\J)))<\bnumber(\geq_\I\cap (\cD_\J\times\cD_\J))$.

Since $\cG\subseteq\cD_\J$ and $|\cG|\leq\cf(\bnumber(\geq_\I\cap (\cD_\K\times\cD_\J))<\bnumber(\geq_\I\cap (\cD_\J\times\cD_\J)$, there is $g\in \cD_\J$ with $g_\alpha\geq_\I g$ for every $\alpha$.
Consequently, $f\geq_\I g$ for every $f\in \cF$, a contradiction.

(\ref{prop:cofinality-of-b-number-cofinality2})
This can be proved in a similar way as item \ref{prop:cofinality-of-b-number-cofinality}.

(\ref{prop:cofinality-of-b-number-regular})
If $\J\subseteq\K$, then $\bnumber(\geq_\I\cap (\cD_\K\times\cD_\J)) \leq \bnumber(\geq_\I\cap (\cD_\J\times\cD_\J))$ and $\bnumber(>_\I\cap (\cD_\K\times\cD_\J)) \leq \bnumber(>_\I\cap (\cD_\J\times\cD_\J))$. 
Now using (\ref{prop:cofinality-of-b-number-cofinality}) and (\ref{prop:cofinality-of-b-number-cofinality2}), we get 
$\bnumber(\geq_\I\cap (\cD_\K\times\cD_\J)) \leq \bnumber(\geq_\I\cap (\cD_\J\times\cD_\J))\leq \cf(\bnumber(\geq_\I\cap (\cD_\K\times\cD_\J)))$ and $\bnumber(>_\I\cap (\cD_\K\times\cD_\J)) \leq \bnumber(>_\I\cap (\cD_\J\times\cD_\J))\leq \cf(\bnumber(>_\I\cap (\cD_\K\times\cD_\J)))$.
\end{proof}


\subsection{The bounding numbers \`{a} la Staniszewski}

\begin{definition}
	Let $\I$ be an ideal on $\omega$.
	By $ \widehat{\cP}_\I$ we will denote the family of all sequences $\langle A_n : n\in\omega\rangle$ such that 
	$A_n\in \I$ for  all $n\in\omega$
	and
	$A_n\cap A_k=\emptyset$ for $n\neq k$.
	By $ \cP_\I$ we will denote the family of all sequences $\langle A_n : n\in\omega\rangle$ such that
	$\langle A_n : n\in\omega\rangle \in \widehat{\cP}_\I$ and 
	$\bigcup\{A_n: n\in \omega\} = \omega$.
\end{definition}

\begin{definition}
	For ideals $\I$, $\J$ and $\K$ on $\omega$
	we define
	\begin{equation*}
		\begin{split}
			\bb(\I,\J,\K)  & = 
			\min \left\{|\cE|:\cE\subseteq\widehat{\cP}_\K \land \forall_{\langle A_n\rangle\in\cP_\J}\, \exists_{\langle E_n\rangle\in\cE} \, \bigcup_{n\in\omega}\left(A_n\cap \bigcup_{i\leq n}E_i\right)\notin\I\right\},\\
			\bnumber_s(\I,\J,\K) & = 
			\min \left\{|\cE|:\cE\subseteq\widehat{\cP}_\K \land \forall_{\langle A_n\rangle\in\cP_\J}\, \exists_{\langle E_n\rangle\in\cE} \, \bigcup_{n\in\omega}\left(A_{n+1}\cap \bigcup_{i\leq n}E_i\right)\notin\I\right\}.
		\end{split}
	\end{equation*}
\end{definition}

\begin{remark}
	A topological space $X$ is a QN-space if it does not distinguish pointwise and quasi-normal convergence of sequences of real-valued continuous functions defined on $X$.
	QN-spaces were introduced by Bukovsk\'{y}, Rec\l{}aw and Repick\'{y} in \cite{MR1129696}, and were  thoroughly examined in papers \cite{MR2463820,MR2778559,MR2294632,MR1129696,MR1815270,MR1477547,MR1800160,MR2280899,MR2881299}.
	The research on ideal-QN-spaces was initiated by Das and Chandra in \cite{MR3038073} and has been continued in \cite{MR3622377,MR3784399,MR3924519,MR3423409}. 
	
	In \cite{MR1129696}, the authors proved that 
	the smallest size of non-QN-space is equal to the bounding number $\bnumber$. 
	The cardinal numbers  $\bb(\I,J,K)$ and $\bb_s(\I,\J,\K)$ were introduced by Filip\'{o}w and Staniszewski~\cite{MR3269495,MR3624786} to characterize the smallest size of a space which is not ideal-QN.
	Recently, Repick\'{y}~\cite{Repicky-1,Repicky-2} thoroughly examined ideal-QN spaces and, among others, characterized the smallest size of non-ideal-QN-spaces in terms of the cardinal $\bnumber(\geq_\I\cap (\cD_\cK\times\cD_\cJ))$. Taking into account Theorem~\ref{thm:Canjar-bNumber-characterization-by-Staniszewski-bNumber}, we see that both Repick\'{y} and Staniszewski obtained the same conclusion but coming from different directions.

	In \cite{MR3423409}, \v{S}upina introduced the cardinal $\kappa(\I,\J)$ which is equal to $\bnumber(\I,\J,\I)$.
\end{remark}

\begin{proposition}
	\label{prop:bNumber-characterization}
	Let $\I, \J$ and $\K$ be ideals on $\omega$.
	\begin{enumerate}
		\item 
		$\bnumber(\I,\J,\K)\leq\bnumber_s(\I,\J,\K)$.\label{prop:bNumber-characterization-inequality}
		
		\item If $\J\cap \K\subseteq \I$, then 
		$\bnumber(\I,\J,\K) = \bnumber_s(\I,\J,\K)$.\label{prop:bNumber-characterization-equality}
		
		\item 
		$\bnumber(\fin,\fin,\fin) = \bnumber_s(\fin,\fin,\fin) =\bnumber$.\label{prop:bNumber-characterization-FIN}
		
	\end{enumerate}
\end{proposition}

\begin{proof}
	(\ref{prop:bNumber-characterization-inequality})
	Obvious. 
	
	\smallskip
	
	(\ref{prop:bNumber-characterization-equality})
	It is enough to show 
	$\bnumber(\I,\J,\K)\geq \bnumber_s(\I,\J,\K)$.
	
	Let $\cE\subseteq\widehat{\cP}_\K$ be such that for every $\langle A_n:n\in\omega\rangle\in\mathcal{P}_\J$ there is $\langle E_n\rangle \in \cE$ with
	$\bigcup_{n\in\omega}(A_n\cap \bigcup_{i\leq n}E_i)\notin\I$.
	
	For every $E = \langle E_n\rangle\in \cE$ we define $F^E_0=E_0\cup E_1$ and $F^E_n=E_{n+1}$ for $n\geq 1$.
	Then $\cF=\{\langle F^E_n\rangle:E\in \cE\}\subseteq\widehat{\cP}_\K$ and $|\cF|\leq |\cE|$.
	
	Let $\langle A_n\rangle\in \cP_\J$.
	Then there is $\langle E_n\rangle \in \cE$ with
	$B = \bigcup_{n\in\omega}(A_n\cap \bigcup_{i\leq n}E_i)\notin\I$.
	
	Since $A_0\cap E_0\in \J\cap \K\subseteq\I$
	and 
	$
	\bigcup_{n\in\omega}(A_{n+1}\cap \bigcup_{i\leq n}F^E_i) = 
	(\bigcup_{n\in\omega}(A_{n}\cap \bigcup_{i\leq n}E_i))\setminus (A_0\cap E_0)
	= B\setminus (A_0\cap E_0)$, we get $\bigcup_{n\in\omega}(A_{n+1}\cap \bigcup_{i\leq n}F^E_i)  
	\notin\I$, and the proof is finished.

	\smallskip

	(\ref{prop:bNumber-characterization-FIN})
	It was proved in \cite{MR3269495,MR3624786}.	
\end{proof}

\begin{proposition}
	\label{prop:bNumber-monotonicity}
	Let $\I,\I',\J,\J',\K,\K'$ be ideals on $\omega$.
	\begin{enumerate}
		\item If $\I\subseteq\I'$, then 
		$\bb(\I,\J,\K)\leq \bb(\I',\J,\K)$
		and
		$\bb_s(\I,\J,\K)\leq \bb_s(\I',\J,\K)$.
		\item If $\J\subseteq\J'$, then 
		$\bb(\I,\J,\K)\leq \bb(\I,\J',\K)$
		and
		$\bb_s(\I,\J,\K)\leq \bb_s(\I,\J',\K)$.
		\item If $\K\subseteq\K'$, then 
		$\bb(\I,\J,\K)\geq \bb(\I,\J,\K')$
		and
		$\bb_s(\I,\J,\K)\geq \bb_s(\I,\J,\K')$.
	\end{enumerate}
\end{proposition}

\begin{proof}
	Straightforward.
\end{proof}

In the sequel we will often use the following equivalent forms of $\bb(\I,\J,\K)$ without any reference.

\begin{proposition}
	\label{prop:bNumber-characterization-partitions}
	Let $\I$, $\J$ and $\K$ be ideals on $\omega$.
	Then $\bb(\I,\J,\K)  = \bb_1(\I,\J,\K)=\bb_2(\I,\J,\K)=\bb_3(\I,\J,\K)$, where 
	\begin{equation*}
		\begin{split}
			\bb_1(\I,\J,\K)  =& 
			\min \left\{|\cE|:\cE\subseteq\cP_\K \land \forall_{\langle A_n\rangle\in\cP_\J}\, \exists_{\langle E_n\rangle\in\cE} \, \bigcup_{n\in\omega}\left(A_n\cap \bigcup_{i\leq n}E_i\right)\notin\I\right\},\\
			\bb_2(\I,\J,\K)  = &
			\min \left\{|\cG|:\cG\subseteq\K^\omega \land \forall_{\langle A_n\rangle \in\cP_\J} \, \exists_{\langle G_n\rangle \in\cG} \, \bigcup_{n\in\omega}\left(A_n\cap G_n\right)\notin\I\right\}, \\
			\bb_3(\I,\J,\K)  =  &
			\min \left\{|\cE|:\cE\subseteq\widehat{\cP}_\K \land \forall_{\langle A_n\rangle\in\cP_\J} \, \exists_{\langle E_n\rangle\in\cE} \, \bigcup_{k\in\omega}\left(E_k\setminus \bigcup_{n<k}A_n\right)\notin\I\right\}. \\
		\end{split}
	\end{equation*}
\end{proposition}

\begin{proof}
	First we show $\bb(\I,\J,\K)=\bb_1(\I,\J,\K)$.
	Since $\bb(\I,\J,\K)\leq\bb_1(\I,\J,\K)$ is obvious, we only show $\bb(\I,\J,\K)\geq \bb_1(\I,\J,\K)$.
	Let 
	$\cE\subseteq\widehat{\cP_\K}$ 
	be such that  for every $\langle A_n\rangle\in\mathcal{P}_\J$ there is $\langle E_n\rangle\in\cE$ with 
	$\bigcup_{n\in\omega}(A_n\cap \bigcup_{i\leq n}E_i)\notin\I$.
	For every $E=\langle E_n\rangle\in \cE$ 
	we define $F^E_n = E_n\cup\{e_n\}$, where 
	$e_1,e_2,\dots$ is an enumeration of 
	$B=\omega\setminus \bigcup_{n\in \omega} E_n$ (if $B$ is finite, we put $F^E_n = E_n$ for $n>|B|$).
	Let $\cF=\{\langle F^E_n\rangle:E\in \cE\}$, and notice that $\cF\subseteq\cP_\K$ and $|\cF|\leq|\cE|$.
	Let $\langle A_n\rangle\in \cP_\J$.
	There is $E=\langle E_n\rangle\in \cE$ with 
	$\bigcup_{n\in\omega}(A_n\cap \bigcup_{i\leq n}E_i)\notin\I$.
	Then 
	$\bigcup_{n\in\omega}(A_n\cap \bigcup_{i\leq n}E_i)
	\subseteq 
	\bigcup_{n\in\omega}(A_n\cap \bigcup_{i\leq n}F^E_i)$, so $\bigcup_{n\in\omega}(A_n\cap \bigcup_{i\leq n}F^E_i)\notin\I$.
	Thus $\bb_1(\I,\J,\K)\leq \bb(\I,\J,\K)$.
	
	
	\smallskip
	
	Now we show $\bb(\I,\J,\K)\leq \bb_2(\I,\J,\K)$.
	Let $\cG\subset\cK^\omega$ such that for every $\langle A_n\rangle \in \cP_\J$ there is $\langle G_n\rangle\in \cG$ with $\bigcup_{n\in\omega}(A_n\cap G_n)\notin\I$.
	For every $G=\langle G_n\rangle\in \cG$ we define $E^G_0=G_0$ and $E^G_{n}=G_n\setminus \bigcup_{i<n}G_i$ for $n\geq 1$.
	Let $\cE = \{\langle E^G_n\rangle: G\in \cG\}$, and notice that $\cE\subseteq \widehat{\cP_\K}$ and $|\cE|\leq |\cG|$.
	Let $\langle A_n\rangle\in \cP_\J$.
	Then there is $G=\langle G_n\rangle\in \cG$ with $\bigcup_{n\in\omega}(A_n\cap G_n)\notin\I$.
	But $\bigcup_{n\in\omega}(A_n\cap G_n)= \bigcup_{n\in\omega}(A_n\cap \bigcup_{i\leq n} E^G_i)$, so $\bnumber(\I,\J,\K)\leq |\cG|$ and the proof of this case is finished.

	
	\smallskip
	
	Now we show $\bb(\I,\J,\K)\geq \bb_2(\I,\J,\K)$.
	Let $\cE\subset\widehat{\cP_\cK}$ such that for every $\langle A_n\rangle \in \cP_\J$ there is $\langle E_n\rangle\in \cE$ with $\bigcup_{n\in\omega}(A_n\cap \bigcup_{i\leq n}E_i)\notin\I$.
	For every $E=\langle E_n\rangle\in \cE$ we define $G^E_{n}=\bigcup_{i\leq n}E_i$ for $n\in\omega$.
	Let $\cG= \{\langle G^E_n\rangle: E\in \cE\}$, and notice that $\cG\subseteq \K^\omega$ and $|\cG|\leq |\cE|$.
	Let $\langle A_n\rangle\in \cP_\J$.
	Then there is $E=\langle E_n\rangle\in \cE$ with $\bigcup_{n\in\omega}(A_n\cap \bigcup_{i\leq n} E_i)\notin\I$.
	But $\bigcup_{n\in\omega}(A_n\cap \bigcup_{i\leq n} E_i) = \bigcup_{n\in\omega}(A_n\cap G_n)$, so $\bnumber_2(\I,\J,\K)\leq |\cE|$ and the proof of this case is finished.
	
	
	\smallskip
	
	Finally, we show $\bb(\I,\J,\K)=\bb_3(\I,\J,\K)$. Actually, it suffices to observe (using the fact that $\langle A_n\rangle$ is a partition of $\omega$) that
	\begin{equation*}
		\begin{split}
			&x\in\bigcup_{n\in\omega}\left(A_n\cap \bigcup_{i\leq n}E_i\right) 
			\iff
			(\exists n\in\omega)(\exists k\leq n)(x\in A_n \land  x\in E_k)
			\iff\\
			&(\exists k\in\omega)(\exists n\geq k)(x\in E_k \land x\in A_n)
			\iff
			(\exists k\in\omega)(\forall n<k)(x\in E_k \land x\notin A_n)\\
			&\iff
			x\in\bigcup_{k\in\omega}\left(E_k\setminus \bigcup_{n<k}A_n\right).\qedhere
		\end{split}
	\end{equation*}
\end{proof}


\subsection{Staniszewski and Canjar bounding numbers are the same}

\begin{theorem}
	\label{thm:Canjar-bNumber-characterization-by-Staniszewski-bNumber}
	Let $\I$,$\J$ and $\K$ be ideals on $\omega$.
	\begin{enumerate}

		\item $\bnumber(>_\I\cap (\cD_\K\times\cD_\J))=\bnumber(\I,\J,\K).$\label{thm:Canjar-bNumber-characterization-by-Staniszewski-bNumber-strict-inequalities}
		
		\item $\bnumber(\geq_\I\cap (\cD_\K\times\cD_\J)) = \bnumber_s(\I,\J,\K).$\label{thm:Canjar-bNumber-characterization-by-Staniszewski-bNumber-nonstrict-inequalities}
	\end{enumerate}
\end{theorem}

\begin{proof}
	(\ref{thm:Canjar-bNumber-characterization-by-Staniszewski-bNumber-strict-inequalities})
	First we show $\bnumber(>_\I\cap (\cD_\K\times\cD_\J))\leq \bnumber(\I,\J,\K)$.
	Let $\cE\subseteq\widehat{\cP_\K}$ be such that for every $\langle A_n:n\in\omega\rangle\in\mathcal{P}_\J$ there is $\langle E_n\rangle \in \cE$ with
	$\bigcup_{n\in\omega}(A_n\cap \bigcup_{i\leq n}E_i)\notin\I$.
	
	For every $E=\langle E_n\rangle \in \cE$ let $f_E\in \omega^\omega$ be such that $f_E^{-1}[\{n\}] = E_n$ for every $n\in \omega$.
	Let $\cF = \{f_E:E\in \cE\}$. Then $\cF\subseteq\cD_\K$ and $|\cF|=|\cE|$. 
	
	Let $g\in \cD_\J$ and $A_n=g^{-1}[\{n\}]$ for every $n\in \omega$. 
	Since $\langle A_n:n\in \omega\rangle\in \cP_\J$, there is $E=\langle E_n\rangle \in \cE$ with 
	$B = \bigcup_{n\in\omega}(A_n\cap \bigcup_{i\leq n}E_i)\notin\I$.
	
	Once we show 
	$B \subseteq \{k\in \omega: f_E(k)\leq g(k)\}$, the proof will be finished in this case.
	Let $k\in B$.
	Then $k\in A_n\cap E_i$ for some $n\in \omega$ and $i\leq n$.
	Hence $g(k) = n \geq i = f_E(k)$.
	
	\smallskip
	
	Now we show $\bnumber(>_\I\cap (\cD_\K\times\cD_\J))\geq \bnumber(\I,\J,\K) $.
	Let $\cF\subseteq \cD_\K$ be such that there is no $g\in\cD_\J$ such that for all $f\in \cF$ we have $f>_\I g$. 
	
	For every $f\in \cF$ and $n\in \omega$ we define $E^f_n = f^{-1}[\{n\}]$, and we notice that $\cE=\{ \langle E^f_n\rangle: f\in \cF\}\subseteq \widehat{\cP_\cK}$ and $|\cE|=|\cF|$.
	
	Let $\langle A_n:n\in \omega\rangle\in \cP_\J$ and $g\in \omega^\omega$ be such that $g^{-1}[\{n\}] = A_n$. Since $g\in \cD_\J$, there is $f\in \cF$ with $\neg(f>_\I g)$.
	Thus, $B  = \{k\in\omega: f(k)\leq g(k)\}\notin\I$. 
	
	Once we show 
	$B \subseteq \bigcup_{n\in\omega}(A_n\cap \bigcup_{i\leq n}E^f_i)$, the proof will be finished in this case.
	
	Let $k\in B$.
	Since $k\in A_{g(k)}\cap E^f_{f(k)}$ and 
	$f(k)\leq g(k)$, 
	$k\in A_{g(k)}\cap \bigcup_{i\leq g(k)}E^f_i$.
	
	\smallskip
	
	(\ref{thm:Canjar-bNumber-characterization-by-Staniszewski-bNumber-nonstrict-inequalities})
	First we show 
	$\bnumber(\geq_\I\cap (\cD_\K\times\cD_\J)) \leq \bnumber_s(\I,\J,\K)$.
	Let $\cE\subseteq\widehat{\cP_\K}$ be such that for every $\langle A_n:n\in\omega\rangle\in\mathcal{P}_\J$ there is $\langle E_n\rangle \in \cE$ with
	$\bigcup_{n\in\omega}(A_{n+1}\cap \bigcup_{i\leq n}E_i)\notin\I$.
	
	For every $E=\langle E_n\rangle \in \cE$ let $f_E\in \omega^\omega$ be such that $f_E^{-1}[\{n\}] = E_n$ for every $n\in \omega$.
	Let $\cF = \{f_E:E\in \cE\}$. Then $\cF\subseteq\cD_\K$ and $|\cF|=|\cE|$. 
	
	Let $g\in \cD_\J$ and $A_n=g^{-1}[\{n\}]$ for every $n\in \omega$. 
	Since $\langle A_n:n\in \omega\rangle\in \cP_\J$, there is $E=\langle E_n\rangle \in \cE$ with 
	$B = \bigcup_{n\in\omega}(A_{n+1}\cap \bigcup_{i\leq n}E_i)\notin\I$.
	
	Once we show 
	$B \subseteq \{k\in \omega: f_E(k)<g(k)\}$, the proof will be finished in this case.
	Let $k\in B$.
	Then $k\in A_{n+1}\cap E_i$ for some $n\in \omega$ and $i\leq n$.
	Hence $g(k) = n+1\geq i+1>i = f_E(k)$.
	
	\smallskip
	
	Now we show 
	$\bnumber(\geq_\I\cap (\cD_\K\times\cD_\J)) \geq \bnumber_s(\I,\J,\K)$.
	Let $\cF\subseteq\cD_\K$ be such that for every $g\in \cD_\J$ there is $f\in \cF$ with $\neg(f\geq_I g)$.
	For every $f\in \cF$ and $n\in \omega$ we define $E^f_n=f^{-1}[\{n\}]$, and we notice that $\cE=\{\langle E^f_n\rangle:f\in \cF\}\subseteq\widehat{\cP_\K}$ and $|\cE|\leq |\cF|$.
	
	Take any $\langle A_k\rangle\in \cP_\J$ and 
	define $g\in \omega^\omega$ by $g(n)=k$ for every $n\in A_k$, $k\in \omega$.
	Since $g\in \cD_\J$, there is $f\in \cF$ with $B = \{n\in \omega:f(n)<g(n)\}\notin\I$.
	Once we show that $B\subseteq \bigcup_{k\in\omega}(A_{k+1}\cap \bigcup_{i\leq k}E^f_i)$, the proof will be finished in this case.
	
	Let $n\in \omega$ be such that $f(n)<g(n)$ and let $k\in \omega$ be such that $n\in A_k$.
	Notice that $k\geq 1$. (Indeed, if $k=0$ then $n\in A_0$, so $g(n)=0$, and consequently $f(n)<0$, a contradiction.) 
	Since $f(n)<g(n)=k$, $n\in \bigcup_{i<k}E^f_i$.
	Let $l=k-1$. Then $n\in A_{l+1}\cap \bigcup_{i\leq l}E^f_i$.
\end{proof}

\begin{corollary}
	\label{cor:Canjar-bNumber-properties}
	Let $\I$,$\J$ and $\K$ be ideals on $\omega$.
	\begin{enumerate}
		\item If $\J\cap \K\subseteq \I$, then 
		$\bnumber(>_\I\cap (\cD_\K\times\cD_\J))= \bnumber(\geq_\I\cap (\cD_\K\times\cD_\J))$.\label{cor:Canjar-bNumber-properties-equality}
		
		\item 
		$\bnumber(\geq_\fin\cap (\cD_\fin\times \cD_\fin)) = \bnumber(>_\fin\cap (\cD_\fin\times \cD_\fin)) = \bnumber$.\label{cor:Canjar-bNumber-properties-FIN}
	\end{enumerate}
\end{corollary}

\begin{proof}	
	(\ref{cor:Canjar-bNumber-properties-equality})
	Follows from
	Theorem~\ref{thm:Canjar-bNumber-characterization-by-Staniszewski-bNumber} and Proposition~\ref{prop:bNumber-characterization}(\ref{prop:bNumber-characterization-equality}).
	
	(\ref{cor:Canjar-bNumber-properties-FIN})
	Follows from
	Theorem~\ref{thm:Canjar-bNumber-characterization-by-Staniszewski-bNumber} and Proposition
	\ref{prop:bNumber-characterization}(\ref{prop:bNumber-characterization-FIN}).
\end{proof}

\begin{corollary}
\label{cor:regular}
	Let $\I$,$\J$ and $\K$ be ideals on $\omega$. If $\J\subseteq\K$, then $\bnumber(\I,\J,\K)$ and $\bnumber_s(\I,\J,\K)$ are regular cardinals.
\end{corollary}

\begin{proof}	
Follows from
	Theorem~\ref{thm:Canjar-bNumber-characterization-by-Staniszewski-bNumber} and Proposition \ref{prop:cofinality-of-b-number}(\ref{prop:cofinality-of-b-number-regular}).
\end{proof}


\subsection{Extreme cases}

\begin{proposition}
	\label{prop:extreme-cases}
	Let $\I,\J,\K$ be ideals on $\omega$.
	\begin{enumerate}
		
		\item 
		$1\leq \bb(\I,\J,\K)\leq \bb_s(\I,\J,\K)\leq \continuum^+$.\label{prop:extreme-cases-geq-1}
		
		\item 
		If $\K\not\subseteq \I$, then $\bb(\I,\J,\K)=1$.\label{prop:extreme-cases-equals-1}
		
		\item 
		If $\J\cap \K\subseteq\I$ and $\K\not\subseteq \I$, then $\bnumber_s(\I,\J,\K)=1$.\label{prop:extreme-cases-equals-1-ALL}
		
		\item 
		If $\K\subseteq\J$, then $\bnumber_s(\I,\J,\K) \geq\aleph_0$.\label{prop:extreme-cases-geq-alef-zero}
		
		\item 
		If $\J\subseteq\K$ and $\I\not\perp \J$, then $\bnumber_s(\I,\J,\K)\leq \continuum$.\label{prop:extreme-cases-leq-continuum-bs}
		
		\item 
		If $\J\subseteq\K$, then $\bnumber(\I,\J,\K)\leq \continuum$.\label{prop:extreme-cases-leq-continuum}
		
		\item 
		If $\I\perp\J$, then 
		$\bb_s(\I,\J,\K)=\continuum^+$.
		\label{prop:extreme-cases-equals-continuum-plus-only-orthogonal}
		
		\item 
		If $\J\cap \K\subseteq\I$ and $\I\perp\J$, then $\bb(\I,\J,\K)=\continuum^+$.
		\label{prop:extreme-cases-equals-continuum-plus-ALL}
	\end{enumerate}
\end{proposition}

\begin{proof}
	
	(\ref{prop:extreme-cases-geq-1})
	Obvious.

	(\ref{prop:extreme-cases-equals-1})
	By (\ref{prop:extreme-cases-geq-1}) it is enough to show that $\bb(\I,\J,\K)\leq 1$.
	Let $E\in \K\setminus \I$ and $\cE = \{\langle E,\emptyset,\emptyset,\dots\rangle\}$.
	Since $E\in \K$, $\cE\in\widehat{\cP_\K}$.
	Let $\langle A_n\rangle\in \cP_\J$.
	Then 
	$
	\bigcup_{n\in\omega}(A_n\cap \bigcup_{i\leq n}E_i) = 
	\bigcup_{n\in\omega}(A_n\cap E) = (\bigcup_{n\in\omega}A_n)\cap E = \omega\cap E=E\notin \I$.
	Consequently, $\bb(\I,\J,\K)\leq |\cE|=1$.

	(\ref{prop:extreme-cases-equals-1-ALL})
	Follows from (\ref{prop:extreme-cases-equals-1})
	and Proposition~\ref{prop:bNumber-characterization}(\ref{prop:bNumber-characterization-equality}).

	(\ref{prop:extreme-cases-geq-alef-zero})
	It follows from \cite[Theorem 4.9(5)]{MR3624786}.

	(\ref{prop:extreme-cases-leq-continuum-bs})
	It follows from \cite[Theorem 4.9(2)]{MR3624786}.

	(\ref{prop:extreme-cases-leq-continuum})
	If $\I\not\perp\J$, it follows from 
	(\ref{prop:extreme-cases-leq-continuum-bs}).
	If $\I\perp\J$, then $\cK\not\subseteq\cI$ in this case, so
	(\ref{prop:extreme-cases-equals-1}) finishes the proof.
	
	(\ref{prop:extreme-cases-equals-continuum-plus-only-orthogonal})
	It follows from \cite[Theorem 4.9(1)]{MR3624786}.

	(\ref{prop:extreme-cases-equals-continuum-plus-ALL})
	Follows from (\ref{prop:extreme-cases-equals-continuum-plus-only-orthogonal}) 
	and Proposition~\ref{prop:bNumber-characterization}(\ref{prop:bNumber-characterization-equality}).
\end{proof}

\begin{example}
	Let 
	$\I=\{A\subseteq\omega: A\cap \{2n:n\in\omega\}\in \fin\}$
	and
	$\J=\K=\{A\subseteq\omega: A\cap \{2n+1:n\in\omega\}\in \fin\}$.
	Then by Proposition~\ref{prop:extreme-cases} we have
	$\bnumber(\I,\J,\K) =1$
	and
	$\bnumber_s(\I,\J,\K)=\continuum^+ .$
\end{example}


\section{Basic relationships between the bounding numbers}
\label{sec:basic-relations-between-the-bounding-numbers}

In the rest of the paper we examine the cardinals $\bnumber(\I,\J,\K)$ and $\bnumber_s(\I,\J,\K)$ only in the case when $|\{\I,\J,\K\}\setminus\{\fin\}|\leq 1$ i.e.~a priori one could consider cardinals for  8 triples:
$(\fin,\fin,\fin)$,
$(\fin,\fin,\I)$,
$(\fin,\I,\fin)$,
$(\fin,\I,\I)$,
$(\I,\fin,\fin)$,
$(\I,\fin,\I)$,
$(\I,\I,\fin)$,
$(\I,\I,\I)$.
However, it will follow from Corollary~\ref{cor:bNumber-and-SbNumber} and Theorem~\ref{thm:bNumber-for-all-ideals} that in fact  we can restrict our study to four nontrivial cases: 
$\bb(\I,\I,\I)$, $\bb(\I,\I,\fin)$, $\bb(\I,\fin,\fin)$ 
and
$\bb_s(\fin,\I,\I)$.

In the second part of this section we present diagrams which summarize most results proved in the first part. 


\subsection{Basic properties and relationships}

\begin{corollary}
\label{cor:bNumber-and-SbNumber}
$\bb(\I,\J,\K)=\bb_s(\I,\J,\K)$ 
for any triple
$(\I,\J,\K)
\in \{(\fin,\fin,\fin),$
$(\fin,\fin,\I),$
$(\fin,\I,\fin),$
$(\I,\fin,\fin),$
$(\I,\fin,\I),$
$(\I,\I,\fin),$
$(\I,\I,\I)\}$.
\end{corollary}

\begin{proof}
It follows from Proposition~\ref{prop:bNumber-characterization}(\ref{prop:bNumber-characterization-equality}).
\end{proof}


\begin{theorem}
\label{thm:bNumber-for-all-ideals}
Let $\I$ be an ideal on $\omega$ such that $\I\neq\fin$.
\begin{enumerate}
\item 
$\bb(\fin,\fin,\I)=\bb(\fin,\I,\I)=1$.\label{thm:bNumber-for-all-ideals:FFI-FII-one}

\item 
$\bb(\fin,\fin,\fin)=\bb(\fin,\I,\fin)=\bnumber$.\label{thm:bNumber-for-all-ideals:FFF-FIF-b}

\item
\begin{enumerate}
\item 
If $\I$ is not a P-ideal, $\bb(\I,\fin,\I)=1$.\label{thm:bNumber-for-all-ideals:IFI-nonP}
\item
If $\I$ is a P-ideal, $\bb(\I,\fin,\I)=\bb(\I,\I,\I)$.\label{thm:bNumber-for-all-ideals:IFI-P}
\end{enumerate}

\item
\begin{enumerate}
	\item 
	If $\I$ is not a weak P-ideal, $\bb(\I,\I,\fin)=\continuum^+$.\label{thm:bNumber-for-all-ideals:IIF-nonWeakP}
	\item
		If $\I$ is a weak P-ideal, $\bb(\I,\I,\fin)\leq\dd$.\label{thm:bNumber-for-all-ideals:IIF-WeakP}
\end{enumerate}
\label{thm:bNumber-for-all-ideals:INEQUALITIES-IIF}

\item 
$\aleph_1\leq \bb(\I,\I,\I) \leq \continuum$.
\label{thm:bNumber-for-all-ideals:INEQUALITIES-III}
\item 
$\bb\leq \bb(\I,\fin,\fin) \leq \dd$.
\label{thm:bNumber-for-all-ideals:INEQUALITIES-IFF}
\item 
$\bb(\I,\fin,\fin)\leq \bb(\I,\I,\fin)$.\label{thm:bNumber-for-all-ideals:INEQUALITIES-1}
\item $\bb(\I,\I,\I)$ and $\bb(\I,\fin,\fin)$ are regular.
\label{thm:bNumber-for-all-ideals:regular}

\item 
$\bb(\I,\I,\I) \leq\cf(\bb(\I,\I,\fin))\leq \bb(\I,\I,\fin)$.\label{thm:bNumber-for-all-ideals:INEQUALITIES-2}
\end{enumerate}
\end{theorem}

\begin{proof}
(\ref{thm:bNumber-for-all-ideals:FFI-FII-one}) 
It follows from Proposition~\ref{prop:extreme-cases}(\ref{prop:extreme-cases-equals-1}).

(\ref{thm:bNumber-for-all-ideals:FFF-FIF-b}) 
It follows from \cite[Proposition 4.5]{MR3269495}.

(\ref{thm:bNumber-for-all-ideals:IFI-nonP})
It follows from \cite[Proposition~4.1]{MR3269495}.

(\ref{thm:bNumber-for-all-ideals:IFI-P})
It follows from \cite[Corollary~4.4]{MR3269495}.

(\ref{thm:bNumber-for-all-ideals:IIF-nonWeakP})
It follows from \cite[Theorem~4.9(1)]{MR3624786}

(\ref{thm:bNumber-for-all-ideals:IIF-WeakP})
It follows from \cite[Corollary 2.10]{MR3784399}.

(\ref{thm:bNumber-for-all-ideals:INEQUALITIES-III})
It follows from  \cite[Proposition 4.1]{MR3269495}.

(\ref{thm:bNumber-for-all-ideals:INEQUALITIES-IFF})
By (\ref{thm:bNumber-for-all-ideals:FFF-FIF-b}) and 
Proposition~\ref{prop:bNumber-monotonicity} we have $\bb=\bb(\fin,\fin,\fin)\leq \bb(\I,\fin,\fin)$.
Now we show $\bb(\I,\fin,\fin)\leq \dd$.
Let $\cF\subset\omega^\omega$ be a dominating family in $(\omega^\omega,\leq^*)$.
Without loss of generality we can assume that every $f\in \cF$ is increasing.
We define $E^f_n = [f(n-1), f(n))\cap \omega$ for every $f\in \cF$, $n\in \omega$ (here we put $f(-1) = 0$).

Notice that $\langle E^f_n:n\in\omega\rangle\in\widehat{\cP}_\fin$
 for every $f\in \cF$.
Let 
$\langle A_n:n\in\omega\rangle\in\mathcal{P}_\fin$.
Let $g\in \omega^\omega$ be given by $g(n) = \max (\bigcup_{i\leq n} A_i)$.
Note that if $g(n)<f(n)$, then $A_n \subseteq \bigcup_{i\leq n} A_i \subseteq [0,f(n)) = \bigcup_{i\leq n}E^f_i$.

Since $\cF$ is dominating, there is $f\in \cF$ with $g<^* f$.
Let $N\in \omega$ be such that $g(n)<f(n)$ for all $n\geq N$.

Then 
$\bigcup_{n\in\omega}(A_n\cap \bigcup_{i\leq n}E^f_i)
\supseteq
\bigcup_{n\geq N}(A_n\cap \bigcup_{i\leq n}E^f_i)
=
\bigcup_{n\geq N}(\bigcup_{i\leq n}E^f_i)
=\omega\notin\I
$,
and
the proof is finished.

(\ref{thm:bNumber-for-all-ideals:INEQUALITIES-1})
It follows from Proposition~\ref{prop:bNumber-monotonicity}.

(\ref{thm:bNumber-for-all-ideals:regular})
It follows from 
Theorem~\ref{thm:Canjar-bNumber-characterization-by-Staniszewski-bNumber}(\ref{thm:Canjar-bNumber-characterization-by-Staniszewski-bNumber-strict-inequalities}, \ref{thm:Canjar-bNumber-characterization-by-Staniszewski-bNumber-nonstrict-inequalities}) and Proposition~\ref{prop:cofinality-of-b-number}(\ref{prop:cofinality-of-b-number-regular}).

(\ref{thm:bNumber-for-all-ideals:INEQUALITIES-2})
The second inequality is obvious and the first inequality follows from 
Theorem~\ref{thm:Canjar-bNumber-characterization-by-Staniszewski-bNumber}(\ref{thm:Canjar-bNumber-characterization-by-Staniszewski-bNumber-strict-inequalities}) and Proposition~\ref{prop:cofinality-of-b-number}(\ref{prop:cofinality-of-b-number-cofinality}).
\end{proof}

\begin{corollary}
$\bb(\I,\J,\K)$ and $\bb_s(\I,\J,\K)$ are regular cardinals for any triple
$(\I,\J,\K)
\in \{(\fin,\fin,\fin),$
$(\fin,\fin,\I),$
$(\fin,\I,\fin),$
$(\fin,\I,\I),$
$(\I,\fin,\fin),$
$(\I,\fin,\I),$
$(\I,\I,\I)\}$.
\end{corollary}

\begin{proof}
By Corollary~\ref{cor:bNumber-and-SbNumber} and Theorem \ref{thm:bNumber-for-all-ideals}(\ref{thm:bNumber-for-all-ideals:FFF-FIF-b}), $\bb(\fin,\I,\fin)=\bb_s(\fin,\I,\fin)=\bb$, which is a regular cardinal. The remaining cases follow from Corollary \ref{cor:regular}.
\end{proof}

We were not able to show that $\bb(\I,\I,\fin)$ is regular. However, the next result shows that some values of $\bb(\I,\I,\fin)$ are prohibited (for instance, this is the case for $\aleph_\omega$, as $\cf(\aleph_\omega)=\omega$).

\begin{corollary}
 $\aleph_1\leq\bb(\I,\I,\I)\leq\cf(\bb(\I,\I,\fin))=\cf(\bb_s(\I,\I,\fin))$  for any ideal $\I$.
\end{corollary}

\begin{proof}
By Proposition~\ref{prop:cofinality-of-b-number}(\ref{prop:cofinality-of-b-number-cofinality}) and Theorem~\ref{thm:Canjar-bNumber-characterization-by-Staniszewski-bNumber} we have $\bb(\I,\I,\I)\leq\cf(\bb(\I,\I,\fin))$, $\cf(\bb(\I,\I,\fin))=\cf(\bb_s(\I,\I,\fin))$ follows from Corollary~\ref{cor:bNumber-and-SbNumber} and $\aleph_1\leq\bb(\I,\I,\I)$ follows from Theorem \ref{thm:bNumber-for-all-ideals}(\ref{thm:bNumber-for-all-ideals:INEQUALITIES-III}).
\end{proof}

As shown in the next result, for P-ideals the situation is much simpler -- many considered cardinals coincide.

\begin{theorem}
\label{thm:bNumber-for-P-ideals}
If $\I$ is a P-ideal, then 
$$\bb\leq \bb(\I,\fin,\fin)=\bb(\I,\I,\I)=\bb(\I,\fin,\I)=\bb(\I,\I,\fin)\leq\dd.$$
\end{theorem}

\begin{proof}
Taking into account the previous results and the fact that every P-ideal is a weak P-ideal, it is enough to show
$\bb(\I,\fin,\fin) \geq \bb(\I,\I,\fin)$
and
$\bb(\I,\I,\I) \geq  \bb(\I,\I,\fin)$.

First we show $\bb(\I,\fin,\fin) \geq \bb(\I,\I,\fin)$.

Let $\cE\in \cP_\fin$ be such that for every $\langle A_n\rangle\in \cP_\fin$ there is $\langle E_n\rangle\in \cE$ with 
$\bigcup_{n\in\omega}(A_n\cap \bigcup_{i\leq n}E_i)\notin\I$.

Let $\langle B_n\rangle\in \cP_\I$.
Since $\I$ is a P-ideal, there is $B\in \I$ with $B_n\setminus B\in \fin$ for every $n\in \omega$.
Let $B=\{b_n:n\in \omega\}$ and set $A_n = (B_n\setminus B)\cup\{b_n\}$ for every $n\in \omega$.
Then $\langle A_n\rangle \in \cP_\fin$, so there is 
$\langle E_n\rangle\in \cE$ with 
$C = \bigcup_{n\in\omega}(A_n\cap \bigcup_{i\leq n}E_i)\notin\I$.

Once we show $C\setminus B \subseteq \bigcup_{n\in\omega}(B_n\cap \bigcup_{i\leq n}E_i)$, the proof of this case will be finished.

Let $k\in C\setminus B$. Then there is $n\in \omega$ with $k\in [A_n\cap \bigcup_{i\leq n}E_i]\setminus B 
= 
[((B_n\setminus B)\cup\{b_n\})\cap \bigcup_{i\leq n}E_i]\setminus B 
=
(B_n\setminus B)\cap \bigcup_{i\leq n}E_i
\subseteq
B_n\cap \bigcup_{i\leq n}E_i
$.


\medskip

Now we show $\bb(\I,\I,\I) \geq  \bb(\I,\I,\fin)$.

Let $\cE\in \cP_\I$ be such that for every $\langle A_n\rangle\in \cP_\I$ there is $\langle E_n\rangle\in \cE$ with 
$\bigcup_{n\in\omega}(A_n\cap \bigcup_{i\leq n}E_i)\notin\I$.

Since $\I$ is a P-ideal, for every $E=\langle E_n\rangle\in \cE$ there is $B_E\in \I$ such that $D_n^E = E_n\setminus B_E\in\fin$ for every $n\in \omega$.
Note that $D_n^E\cap D_k^E=\emptyset$ for every $E\in \cE$ and $n\neq k$.

Let $\langle A_n\rangle\in \cP_\I$.
Then there is $E=\langle E_n\rangle\in \cE$ with 
$C = \bigcup_{n\in\omega}(A_n\cap \bigcup_{i\leq n}E_i)\notin\I$.

Once we show $C\setminus B_E \subseteq \bigcup_{n\in\omega}(A_n\cap \bigcup_{i\leq n}D^E_i)$, the proof of this case will be finished.

Let $k\in C\setminus B_E$. 
Then there is $n\in\omega$ with 
$k
\in 
[\bigcup_{n\in\omega}(A_n\cap \bigcup_{i\leq n}E_i)]\setminus B_E
=
\bigcup_{n\in\omega}(A_n\cap \bigcup_{i\leq n}(E_i\setminus B_E))
=
\bigcup_{n\in\omega}(A_n\cap \bigcup_{i\leq n}D_i^E)
$.
\end{proof}

The rest of this subsection is devoted to $\bb_s(\fin,\I,\I)$. 

\begin{lemma}
	\label{lem:SbNumber-for-all-ideals}
	Let $\I$ be an ideal on $\omega$.
		\begin{enumerate}
			\item 
			If $\I$ is not a P-ideal, then $\bb_s(\fin,\I,\I)=\aleph_0$.\label{lem:SbNumber-for-all-ideals:FII-nonPideal}
			\item If $\I$ is a P-ideal, then 
			$\aleph_1\leq \bb_s(\fin,\I,\I)\leq\bb$.\label{lem:SbNumber-for-all-ideals:FII-Pideal}
		\end{enumerate}
\end{lemma}

\begin{proof}
		(\ref{lem:SbNumber-for-all-ideals:FII-nonPideal})
	By Theorem~\ref{prop:extreme-cases}(\ref{prop:extreme-cases-geq-alef-zero}), $\bb_s(\fin,\I,\I)\geq \aleph_0$. The revers inequality
	follows from 
	\cite[Theorem 4.9(3)]{MR3624786}.
	
	(\ref{lem:SbNumber-for-all-ideals:FII-Pideal})
	The first inequality follows from \cite[Theorem 4.9(4)]{MR3624786}, whereas the second one follows from Theorem~\ref{thm:bNumber-for-all-ideals}(\ref{thm:bNumber-for-all-ideals:FFF-FIF-b})  and the observation  that $\bb_s(\fin,\I,\I)\leq\bb_s(\fin,\I,\fin)=\bb(\fin,\I,\fin)=\bb$.	
\end{proof}

\begin{lemma}
\label{lem:SbNumber-for-all-ideals:FII-leq-addStar}
$\bb_s(\fin,\I,\I)\leq \adds(\I)$ for any ideal $\I$.	
\end{lemma}

\begin{proof}
If $\adds(\I)=\continuum^+$, we are done. Assume that $\adds(\I)\leq \continuum$.
Then there is a family $\cF=\{F_\alpha:\alpha<\adds(\I)\}\subseteq\I$ such that for every $A\in \I$ there is $\alpha$ with $|F_\alpha \setminus A|=\aleph_0$.
For every $\alpha$, we define $E^\alpha_0=F_\alpha$ and $E^\alpha_n=\emptyset$ for $n\geq 1$.
Then $\langle E^\alpha_n:n\in \omega\rangle\in \widehat{\cP}_\I$.

Take any $\langle A_n\rangle\in\cP_\I$.
If we show that $\bigcup_{n\in \omega}(A_{n+1}\cap \bigcup_{i\leq n} E^\alpha_i)\notin\fin$ for some $\alpha$, the proof will be finished.

Since $A_0\in \I$, there is $\alpha$ with $|F_\alpha\setminus A_0|=\aleph_0$.
On the other hand, it is easy to see that 
$F_\alpha\setminus A_0 = E^\alpha_0\setminus A_0 \subseteq \bigcup_{n\in \omega}(A_{n+1}\cap \bigcup_{i\leq n} E^\alpha_i)$, so we are done.
\end{proof}

\begin{theorem}
	\label{thm:SbNumber-min-b-addStar}
$\bb_s(\fin,\I,\I)=\min\{\bnumber,\adds(\I)\}$
for any ideal $\I$.
\end{theorem}

\begin{proof}
	If $\I$ is not a P-ideal, then $\bb_s(\fin,\I,\I)=\aleph_0$ (by Lemma~\ref{lem:SbNumber-for-all-ideals}(\ref{lem:SbNumber-for-all-ideals:FII-nonPideal})) and $\adds(\I)=\aleph_0$. Since $\bb>\aleph_0$, we obtain $\bb_s(\fin,\I,\I)=\min\{\bnumber,\adds(\I)\}$.
	
Assume that $\I$ is a P-ideal.
The inequality ``$\leq$'' follows from Lemma~\ref{lem:SbNumber-for-all-ideals}(\ref{lem:SbNumber-for-all-ideals:FII-Pideal}) and Lemma~\ref{lem:SbNumber-for-all-ideals:FII-leq-addStar}. Below we show the inequality ``$\geq$''. 

Take $\kappa<\min\{\bnumber,\adds(\I)\}$. We will show that $\kappa<\bnumber_s(\fin,\I,\I)$.

Let $\langle E^\alpha_n:n\in\omega\rangle\in \widehat{\cP}_\I$ for $\alpha<\kappa$.
Since $\kappa\cdot \aleph_0 <\adds(\I)$, there is $A\in \I$ such that $|E^\alpha_n\setminus A|<\aleph_0$ for every $\alpha<\kappa$ and $n\in \omega$.

Let 
$B = \omega\setminus A$.
Since 
$\langle E^\alpha_n\cap B\rangle\in\widehat{\cP}_{\fin(B)}$
for every $\alpha<\kappa$
and
$\kappa<\bb=\bb(\fin(B),\fin(B),\fin(B))$, there is
$\langle B_n\rangle\in\cP_{\fin(B)}$
such that 
$$\bigcup_{n\in \omega}\left(B_{n}\cap \bigcup_{i\leq n} \left(E^\alpha_i\cap B\right)\right)\in \fin(B)$$
for every $\alpha<\kappa$.

Let $A_0=A$ and $A_n=B_{n-1}$ for $n\geq 1$.
Then $\langle A_n\rangle\in\cP_\I$
and
for every $\alpha<\kappa$ we have
$$
\bigcup_{n\in \omega}\left(A_{n+1}\cap \bigcup_{i\leq n} E^\alpha_i\right) 
= 
\bigcup_{n\in \omega}\left(B_{n}\cap \bigcup_{i\leq n} E^\alpha_i\right)
=
\bigcup_{n\in \omega}\left(B_{n}\cap \bigcup_{i\leq n} (E^\alpha_i\cap B)\right)\in \fin.$$
\end{proof}


\subsection{Diagrams}

In this subsection we present some diagrams which show inequalities  proved above. In all diagrams, ``$A\to B$'' means ``$A\leq B$''. Moreover, to make diagrams fit into the page, we sometimes write ``$\mathrm{F}$'' instead of ``$\mathrm{Fin}$''.

\begin{figure}[H]
	\begin{tikzcd}
		\mbox{\ }
		&
		\bb(\finShort,\I,\I)=1 \to \bb_s(\finShort,\I,\I)  \arrow[r]\arrow[rd]
		&
		\bb(\finShort,\I,\finShort)=\bb   \arrow[rd]
		&
		\mbox{\ }
		\\
		\bb(\finShort,\finShort,\I)=1 \arrow[r] \arrow[ru]\arrow[rd]
		&
		\bb(\finShort,\finShort,\finShort)=\bb\arrow[rd]\arrow[ru]
		&
		\bb(\I,\I,\I)\arrow[r]
		&
		\bb(\I,\I,\finShort)
		\\
		\mbox{\ }
		& 
		\bb(\I,\finShort,\I)\in\{1,\bb(\I,\I,\I)\}\arrow[r]\arrow[ru]
		&
		\bb(\I,\finShort,\finShort)\arrow[ru]
		&
		\mbox{\ }
	\end{tikzcd}
	\caption{All bounding numbers, including the trivial ones, and inequalities  which follow only from Proposition~\ref{prop:bNumber-monotonicity}, Theorem~\ref{thm:bNumber-for-all-ideals} and Lemma~\ref{lem:SbNumber-for-all-ideals}.}
\end{figure}
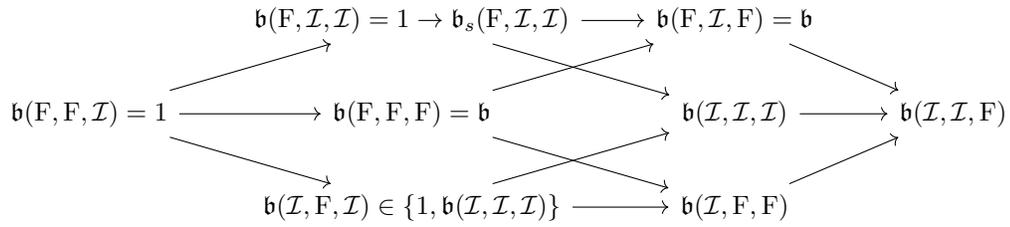

\begin{figure}[H]
	\begin{tikzcd}
		\aleph_1 \arrow[r]\arrow[rrd] 
		&
		\bb \arrow[r] 
		& 
		\bb(\I,\finShort,\finShort) \arrow[r] \arrow[rrd] 
		& 
		\dd \arrow[r] 
		& 
		\continuum 
		\\
		\bb_s(\finShort,\I,\I)=\min\{\bnumber,\adds(\I)\} \arrow[rr] \arrow[ru]\arrow[dr]
		&
		\mbox{\ } 
		&
		\bb(\I,\I,\I) \arrow[rr] \arrow[rru, shift right] 
		&
		\mbox{\ }
		&
		\bb(\I,\I,\finShort) 
\\
		\mbox{\ }
&
\adds(\I)
&
		\mbox{\ }
&
\mbox{\ }
&
\mbox{\ }
	\end{tikzcd}
	\caption{Nontrivial bounding numbers and inequalities valid for arbitrary ideals.}
\end{figure}
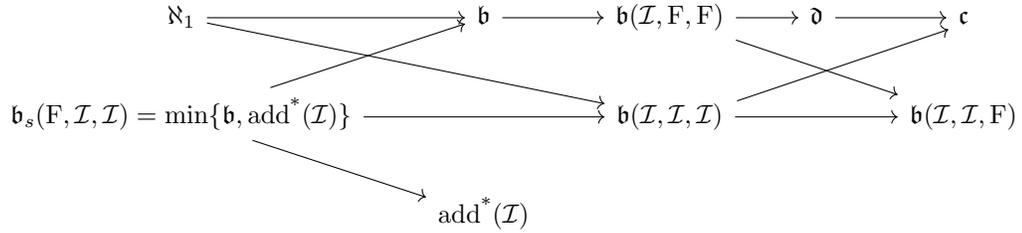

\begin{figure}[H]
	\centering
	\begin{tikzcd} 
		\aleph_1 \arrow[r]\arrow[rrd] 
		&
		\bb \arrow[r] 
		& 
		\bb(\I,\fin,\fin) \arrow[r] 
		& 
		\bb(\I,\I,\fin) \arrow[r] 
		& 
		\dd  
		\\
		&
				\mbox{\ }
		\mbox{\ }
		&
		\bb(\I,\I,\I)  \arrow[ru] 
		&
		\mbox{\ } 
	\end{tikzcd}
	\caption{Inequalities valid for weak P-ideals.}
\end{figure}
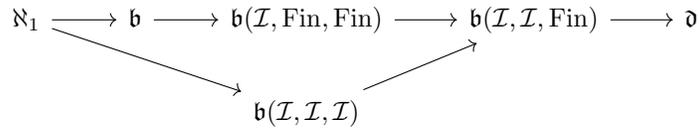

\begin{figure}[H]
	\centering 
	\begin{tikzcd} 
		\aleph_1 \arrow[r]\arrow[rrd] 
		& 
		\bb \arrow[r] 
		&  
		\bb(\I,\fin,\fin) \arrow[r] 
		& 
		\dd \arrow[r] 
		& 
		\continuum 
		\\
		\mbox{\ } 
		&
		\mbox{\ } 
		&
		\bb(\I,\I,\I) \arrow[rru,shift right] 
		&
		\mbox{\ } 
	\end{tikzcd}
	\caption{Inequalities valid for non weak P-ideals.}
\end{figure}
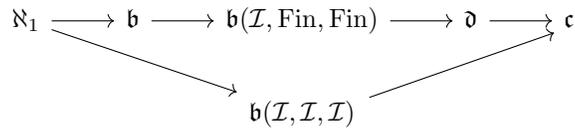

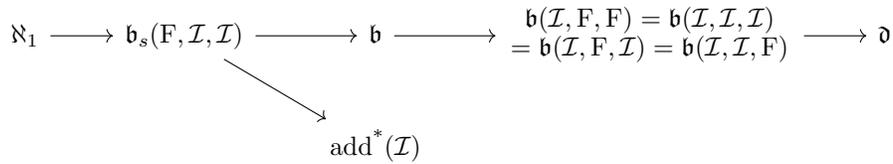
\begin{figure}[H]
\centering
\begin{tikzcd} 
\aleph_1\arrow[r]
& 
\bb_s(\finShort,\I,\I)\arrow[r]  \arrow[rd]
&
\bb 
\arrow[r] 
& 
\genfrac{}{}{0pt}{}{\displaystyle \bb(\I,\finShort,\finShort)
=
\bb(\I,\I,\I)}{\displaystyle =
\bb(\I,\finShort,\I)  
=
\bb(\I,\I,\finShort) }
\arrow[r] 
& 
\dd  
\\
\mbox{\ } 
	&
\mbox{\ } 
	&
\adds(\I)
	&
\mbox{\ } 
	&
\mbox{\ } 
	&
\mbox{\ } 
\end{tikzcd}
\caption{Inequalities valid for P-ideals.}
\end{figure}


\section{Sums and products of ideals}
\label{sec:Sums-and-products-of-ideals}

In this section we examine the considered cardinals for ideals of the form $\I\oplus\J$ or $\I\otimes\J$. Additionally, we compute the considered cardinals in the case of not tall ideals.


\subsection{Disjoint sums}

\begin{theorem}
\label{oplus}
Let $\I$ and $\J$ be ideals on $\omega$.
\begin{enumerate}
\item $\bb(\I\oplus\J,\I\oplus\J,\I\oplus\J)=\min\{\bb(\I,\I,\I),\bb(\J,\J,\J)\}$.
\item $\bb(\I\oplus\J,\I\oplus\J,\fin)=\min\{\bb(\I,\I,\fin),\bb(\J,\J,\fin)\}$.\label{oplus-IIF}
\item $\bb(\I\oplus\J,\fin,\fin)=\min\{\bb(\I,\fin,\fin),\bb(\J,\fin,\fin)\}$,
\item $\bb_s(\fin,\I\oplus\J,\I\oplus\J)=\min\{\bb_s(\fin,\I,\I),\bb_s(\fin,\J,\J)\}$.
\end{enumerate} 
\end{theorem}

\begin{proof}
We will only prove the first item. 
The remaining items can be proved in a similar way.

First we prove that $\bb(\I\oplus\J,\I\oplus\J,\I\oplus\J)\leq\min\{\bb(\I,\I,\I),\bb(\J,\J,\J)\}$. We will prove that $\bb(\I\oplus\J,\I\oplus\J,\I\oplus\J)\leq\bb(\I,\I,\I)$ (the proof of $\bb(\I\oplus\J,\I\oplus\J,\I\oplus\J)\leq\bb(\J,\J,\J)$ goes in a similar way).

There is a family $\{\langle E^\alpha_n:n\in\omega\rangle: \alpha<\bb(\I,\I,\I)\}\subseteq\widehat{\cP}_\I$ 
with the property that for every $\langle A_n:n\in\omega\rangle\in\mathcal{P}_{\I}$ there is $\alpha<\bb(\I,\I,\I)$ such that 
$\bigcup_{n\in\omega}(A_n\cap \bigcup_{i\leq n}E^\alpha_i)\notin\I$.

Consider the family $\{\langle E^\alpha_n\times \{0\}: n\in\omega\rangle:\alpha<\bb(\I,\I,\I)\}\subseteq\widehat{\cP}_{\I\oplus\J}$. Let $\langle B_n\oplus C_n:n\in\omega\rangle\in\mathcal{P}_{\I\oplus\J}$. Since $\langle B_n:n\in \omega\rangle\in \cP_{\I}$, there is $\alpha<\bb(\I,\I,\I)$ such that 
$C = \bigcup_{n\in\omega}(B_n\cap \bigcup_{i\leq n}E^\alpha_i)\notin\I$.
Then $C\times\{0\}\notin \I\oplus\J$, 
and $
C \times\{0\}\subseteq
\bigcup_{n\in\omega}(B_n\oplus C_n\cap \bigcup_{i\leq n}(E^\alpha_i\times\{0\}))
$.
That finishes this part of the proof.

Now we prove $\bb(\I\oplus\J,\I\oplus\J,\I\oplus\J)\geq\min\{\bb(\I,\I,\I),\bb(\J,\J,\J)\}$. Let $\kappa<\min\{\bb(\I,\I,\I),\bb(\J,\J,\J)\}$ and $\{\langle E^\alpha_n\oplus F^\alpha_n: n\in\omega\rangle: \alpha<\kappa\}\subseteq \widehat{\cP}_{\I\oplus\J}$. Since $\kappa<\bb(\I,\I,\I)$ and $\{\langle E^\alpha_n: n\in\omega\rangle :\alpha<\kappa\}\subseteq\widehat{\cP}_{\I}$, there is $\langle B_n:n\in\omega\rangle\in\mathcal{P}_{\I}$ such that $\bigcup_{n\in\omega}(B_n\cap \bigcup_{i\leq n}E^\alpha_i)\in\I$ for all $\alpha<\kappa$. Similarly, there is $\langle C_n:n\in\omega\rangle\in\mathcal{P}_{\J}$ such that $\bigcup_{n\in\omega}(C_n\cap \bigcup_{i\leq n}F^\alpha_i)\in\J$ for all $\alpha<\kappa$.
Then 
$\langle B_n\oplus C_n:n\in \omega\rangle\in\cP_{\I\oplus\J}$ and 
$\bigcup_{n\in\omega}((B_n\oplus C_n)\cap \bigcup_{i\leq n}(E^\alpha_i\oplus F^\alpha_i))\in\I\oplus\J$ for all $\alpha<\kappa$. Thus, $\kappa<\bb(\I\oplus\J,\I\oplus\J,\I\oplus\J)$.
\end{proof}


\subsection{Not tall ideals}

Using the results of the previous subsection, we are able to compute the considered cardinals for not tall ideals.

\begin{theorem}
\label{thm:bNumber-III-less-b-equal-bNumberIFF-IIF-for-not-tall-ideals}
If $\I$ is not a tall ideal, then 
$$\aleph_1\leq \bb(\I,\I,\I) \leq \bb  = \bb(\I,\fin,\fin)=\bb(\I,\I,\fin).$$
\end{theorem}

\begin{proof}
By Theorem~\ref{thm:bNumber-for-all-ideals},
$\bb(\I,\I,\I)\leq \bb(\I,\I,\fin)$ and $\bb\leq \bb(\I,\fin,\fin)\leq \bb(\I,\I,\fin)$, so it is enough to show $\bb(\I,\I,\fin)\leq\bb$.

Let $A\subseteq\omega$ be infinite such that for every $B\subseteq A$, $B\in \I \iff B\text{ is finite}$. Let $\fin(A)$ denote the family of all finite subsets of $A$. Then $\I=\fin(A)\oplus(\I\restriction (\omega\setminus A))$ and, by Theorem~\ref{oplus}(\ref{oplus-IIF}), we have $\bb(\I,\I,\fin)\leq\bb(\fin(A),\fin(A),\fin)=\bb(\fin,\fin,\fin)=\bb$.
\end{proof}

\begin{corollary}
\label{thm:bNumber-for-not-tall-P-ideals}
If $\I$ is a not tall P-ideal, then
$$\aleph_1\leq \bb_s(\fin,\I,\I)\leq \bb=\bb(\I,\fin,\fin)=\bb(\I,\I,\I)=\bb(\I,\I,\fin).$$
\end{corollary}

\begin{proof}
Follows from Theorems~\ref{thm:bNumber-for-P-ideals} and \ref{thm:bNumber-III-less-b-equal-bNumberIFF-IIF-for-not-tall-ideals}.
\end{proof}

\begin{example}
If $\I = \{\emptyset\}\otimes\fin$, then
$$\bb = \bb_s(\fin,\I,\I)  =\bb(\I,\fin,\fin)=\bb(\I,\I,\I)=\bb(\I,\I,\fin).$$
Indeed, it is known that $\adds(\I)=\bb$ (see e.g.~\cite[p.~43]{MR2777744}), so $\bb_s(\fin,\I,\I)=\bb$.
The remaining equalities follows from   Corollary~\ref{thm:bNumber-for-not-tall-P-ideals}, because $\I$ is a P-ideal (see \cite[Example~1.2.3(b)]{MR1711328}) and it 
is not tall (indeed, if $A=\{0\}\times \omega$, then for every $B\subseteq A$, $B\in   \{\emptyset\}\otimes\fin \iff B \text{ is finite}$).
\end{example}


\subsection{Fubini products}

In this subsection we compute the considered cardinals for Fubini products of ideals. We state the lemmas in a more general form than needed in this section, because we will apply some of them also in other sections.

\subsubsection{The cardinal $\bb_s(\fin,\I,\I)$.}

\begin{theorem}
\label{thm:bNumber-for-Fubini-products-FII}
	Let $\J,\K$ be ideals on $\omega$.
		\begin{enumerate}
		\item If $\I=\J\otimes\K$ or  $\I=\J\otimes\{\emptyset\}$, then $\bb_s(\fin,\I,\I)=\aleph_0$.\label{thm:bNumber-for-Fubini-products-FII-notP}
		\item If $\I=\{\emptyset\}\otimes\K$, then $\bb_s(\fin,\I,\I)=\bb_s(\fin,\K,\K)$.
	\end{enumerate}
\end{theorem}

\begin{proof}
(1) It suffices to observe that $\J\otimes\K$ and $\J\otimes\{\emptyset\}$ are not P-ideals.

(2) First we will show $\bb_s(\fin,\I,\I)\leq\bb_s(\fin,\K,\K)$. 

If $\adds(\K)=\cc^+$, then $\bb_s(\fin,\I,\I)\leq\bb=\bb_s(\fin,\K,\K)$ by Theorem \ref{thm:SbNumber-min-b-addStar}. So assume that $\adds(\K)\leq\cc$ and let $\mathcal{F}\subseteq\K$ be such that $|\cF|=\adds(\cK)$ and for any $A\in\K$ there is $F\in\mathcal{F}$ such that $F\setminus A\notin\fin$. 

Consider the family $\left\{\{0\}\times F: F\in\mathcal{F}\right\}\subseteq\I$. We claim that it witnesses $\adds(\I)\leq\adds(\K)$. 

Fix $A\in\I$. Then $(A)_0\in\K$, so there is $F\in\mathcal{F}$ such that $|F\setminus (A)_0|=\aleph_0$. We get $|\{0\}\times F\setminus A|=|F\setminus (A)_0|=\aleph_0$. Thus $\adds(\I)\leq\adds(\K)$ which implies $\bb_s(\fin,\I,\I)\leq\bb_s(\fin,\K,\K)$ by Theorem \ref{thm:SbNumber-min-b-addStar}.

Now we show $\bb_s(\fin,\I,\I)\geq\bb_s(\fin,\K,\K)$.

If $\K$ is not a P-ideal, then $\I$ is not a P-ideal as well and we have $\bb_s(\fin,\I,\I)=\bb_s(\fin,\K,\K)=\aleph_0$. So we can assume that $\K$ is a P-ideal.

Let $\kappa<\bb_s(\fin,\K,\K)$. We will show that $\kappa<\adds(\I)$. By Theorem \ref{thm:SbNumber-min-b-addStar}, it will finish the proof.

Fix $\left\{A^\alpha: \alpha<\kappa\right\}\subseteq\I$. Then $\left\{(A^\alpha)_n: n\in\omega,\alpha<\kappa\right\}\subseteq\K$ has cardinality $|\omega\cdot\kappa|<\bb_s(\fin,\K,\K)\leq\adds(\K)$, so there is $A\in\K$ such that $(A^\alpha)_n\setminus A\in\fin$ for all $n\in\omega$ and $\alpha<\kappa$.

Define $f_\alpha(n)=\max((A^\alpha)_n\setminus A)$ for all $n\in\omega$ and $\alpha<\kappa$. Since $\kappa<\bb_s(\fin,\K,\K)\leq\bb$, there is $g\in\omega^\omega$ such that $f_\alpha\leq^* g$ for all $\alpha<\kappa$. Define $B\subseteq\omega\times\omega$ by $(B)_n=A\cup\{0,1,\ldots,g(n)\}$. Observe that $B\in\I$. If we will show that $A^\alpha\setminus B$ is finite for all $\alpha<\kappa$, the proof will be finished.

Fix $\alpha<\kappa$. There is $k\in\omega$ such that $f_\alpha(n)\leq g(n)$ for all $n\geq k$. Then $(A^\alpha\setminus B)_{n}\subseteq (A^\alpha)_{n}\setminus A$ is finite for all $n<k$ and $(A^\alpha\setminus B)_{n}=(A^\alpha)_{n}\setminus (B)_{n}=(A^\alpha)_{n}\setminus (A\cup\{0,1,\ldots,g(n)\})=\emptyset$ for all $n\geq k$. Hence $A^\alpha\setminus B$ is finite.
\end{proof}

\subsubsection{The cardinal $\bb(\I,\fin,\fin)$.}

The following three lemmas will be useful in our considerations. 

\begin{lemma}
	\label{lemma:Fubini-product-inclusion}
	Let $\langle B_n\rangle$ and $\langle C_n\rangle$ be partitions of $\omega$. If $C_n\subseteq\bigcup_{m\geq n}B_m$ (i.e., $C_n\cap\bigcup_{m<n}B_m=\emptyset$) for all $n\in\omega$, then $\bigcup_{n\in\omega}(C_n\cap \bigcup_{i\leq n}A_i)\subseteq\bigcup_{n\in\omega}(B_n\cap \bigcup_{i\leq n}A_i)$ for any $\langle A_n\rangle$.
\end{lemma}

\begin{proof}
	If $x\in C_n\cap \bigcup_{i\leq n}A_i$ for some $n\in\omega$, then there is $m\geq n$ with $x\in B_m$ and we have $x\in B_m\cap \bigcup_{i\leq n}A_i\subseteq B_m\cap \bigcup_{i\leq m}A_i$.
\end{proof}

\begin{lemma}
\label{lemma:Fubini-product-bounded-from-above-by-second-ideal-IFF}
    Let $\J,\K$ be ideals on $\omega$ 	
    (here we also allow $\J=\{\emptyset\}$)
    and $\I$ be an ideal on $\omega^2$.
	If $\I\subseteq\J\otimes\K$, then $\bb(\I,\fin(\omega^2),\fin(\omega^2))\leq\bb(\K,\fin(\omega),\fin(\omega))$.
\end{lemma}

\begin{proof}
	Let $\{\langle E^\alpha_n:n\in\omega\rangle: \alpha<\kappa\}\subseteq\cP_{\fin(\omega)}$ witness $\kappa=\bb(\K,\fin,\fin)$.
	
	For each $n\in\omega$ and $\alpha<\kappa$ define $A^\alpha_n=(\{0,\ldots,n-1\}\times E^\alpha_n)\cup(\{n\}\times\bigcup_{i\leq n}E^\alpha_i)$. Then $\cA=\{\langle A^\alpha_n:n\in\omega\rangle: \alpha<\kappa\}\subseteq\cP_{\fin(\omega^2)}$.
	
	We will show that $\cA$ witnesses $\bb(\I,\fin(\omega^2),\fin(\omega^2))\leq \kappa$.
	
	Let $\langle B_n:n\in\omega\rangle\in\cP_{\fin(\omega^2)}$. Then 
	$\langle (B_n)_i:n\in\omega\rangle\in\cP_{\fin(\omega)}$ for each $i\in\omega$.
	
	Now we inductively define 
	$$C_n=\left(\bigcup_{i\leq n}((B_n)_i\cup (B_i)_n)\right)\setminus\left(\bigcup_{i<n}C_i\right),$$
	and observe that $\langle C_n:n\in\omega\rangle\in\cP_{\fin(\omega)}$. Therefore, there is $\alpha<\kappa$ such that $\bigcup_{n\in\omega}(C_n\cap \bigcup_{i\leq n}E^\alpha_i)\notin\K$. 
	
If we show that $(\bigcup_{n\in\omega}(B_n\cap \bigcup_{i\leq n}A^\alpha_i))_j\notin \cK$
for every $j\in \omega$, then $\bigcup_{n\in\omega}(B_n\cap \bigcup_{i\leq n}A^\alpha_i)\notin\I$, and the proof will be finished.

	Fix $j\in\omega$ and note that if $n>j$ then
 $(\bigcup_{i\leq n}A^\alpha_i)_j =(\bigcup_{i\leq n}E^\alpha_i)_j$, and
$(B_n)_j\cap \bigcup_{i\leq n}(E^\alpha_i)_j
\supseteq
C_n\cap \bigcup_{i\leq n}(E^\alpha_i)_j$ (the inclusion follows from Lemma~\ref{lemma:Fubini-product-inclusion} because   $C_n\cap\bigcup_{m<n}(B_m)_j=\emptyset$).
Thus, we have
\begin{equation*}
\begin{split}
&	
\left(
\bigcup_{n\in\omega}
\left(B_n\cap \bigcup_{i\leq n}A^\alpha_i\right)
\right)_j 
\supseteq
\left(
\bigcup_{n>j}\left(B_n\cap \bigcup_{i\leq n}A^\alpha_i\right)
\right)_j
=
\bigcup_{n>j}\left((B_n)_j\cap \left(\bigcup_{i\leq n}A^\alpha_i\right)_j\right) \\
&=
\bigcup_{n>j}\left((B_n)_j\cap \left(\bigcup_{i\leq n}E^\alpha_i\right)_j\right)
\supseteq
\bigcup_{n>j}\left(C_n\cap \left(\bigcup_{i\leq n}E^\alpha_i\right)_j \right)\notin\K.
\end{split}
\end{equation*}
\end{proof}

\begin{lemma}
\label{lemma:Fubini-product-bounded-form-below-by-second-ideal-IFF}
Let $\cK$ be an ideal on $\omega$ and $\I$ be an ideal on $\omega^2$.
	If $\{\emptyset\}\otimes\K\subseteq\I$, then $\bb(\I,\fin(\omega^2),\fin(\omega^2))\geq\bb(\K,\fin(\omega),\fin(\omega))$.
\end{lemma}

\begin{proof}
	We will show that $\kappa<\bb(\K,\fin(\omega),\fin(\omega))$ implies $\kappa<\bb(\I,\fin(\omega^2),\fin(\omega^2))$. Let $\kappa<\bb(\K,\fin,\fin)$ and $\{\langle E^\alpha_n:n\in\omega\rangle:\alpha<\kappa\}\subseteq\cP_{\fin(\omega^2)}$.
	
	For each $n,j\in\omega$ and $\alpha<\kappa$ define $A^{\alpha,j}_n= (E^\alpha_n)_j$.
	Then $\{\langle A^{\alpha,j}_n:n\in\omega\rangle:j\in\omega,\alpha<\kappa\}\subseteq\cP_{\fin(\omega)}$.
	
	Since  $|\omega\times\kappa|=\kappa<\bb(\K,\fin(\omega),\fin(\omega))$, there is $\langle B_n\rangle\in\cP_{\fin(\omega)}$ with $\bigcup_{n\in\omega}(B_n\cap \bigcup_{i\leq n}A^{\alpha,j}_i)\in\K$ for all $j\in\omega$ and $\alpha<\kappa$. 
	
	Define $\langle C_n\rangle\in\cP_{\fin(\omega^2)}$ by $C_n=(\bigcup_{i<n}\{i\}\times B_n)\cup(\{n\}\times\bigcup_{i\leq n}B_i)$ for all $n\in\omega$.

	Fix $\alpha<\kappa$ and denote $X=\bigcup_{n\in\omega}(C_n\cap \bigcup_{i\leq n}E^\alpha_i)$. 
	If we show that $X\in \I$, the proof will be finished. Since 
$\{\emptyset\}\otimes\K\subseteq\I$, it is enough to show that $(X)_j\in \cK$ for every $j\in \omega$.
	
	Take $j\in\omega$.
	If 
	$n>j$, then
	$(C_n)_j=B_n$, so 
	$$(X)_j=
	\bigcup_{n\in\omega}\left((C_n)_j\cap \bigcup_{i\leq n}(E^\alpha_i)_j\right)
	\subseteq
	\left(\bigcup_{n\leq j}B_n\right)
	\cup
	\bigcup_{n>j}\left(B_n\cap\bigcup_{i\leq n}A^{\alpha,j}_i\right)\in\K.$$
\end{proof}

\begin{theorem}
		\label{thm:bNumber-for-Fubini-products-IFF}
	Let $\J,\K$ be ideals on $\omega$.
		\begin{enumerate}
		\item If $\I=\J\otimes\K$ or $\I=\{\emptyset\}\otimes\K$, then $\bb(\I,\fin,\fin)=\bb(\K,\fin,\fin)$.
		\label{thm:bNumber-for-Fubini-products:IFF}
		\item If $\I=\J\otimes\{\emptyset\}$, then $\bb(\I,\fin,\fin)=\bb$.
		\label{thm:bNumber-for-Fubini-products:IFF2}
	\end{enumerate}
\end{theorem}

\begin{proof}\ 
	(\ref{thm:bNumber-for-Fubini-products:IFF})
	The inequality ``$\leq$'' follows from 
	Lemma~\ref{lemma:Fubini-product-bounded-from-above-by-second-ideal-IFF}, and  ``$\geq$'' follows from Lemma~\ref{lemma:Fubini-product-bounded-form-below-by-second-ideal-IFF}.
	(\ref{thm:bNumber-for-Fubini-products:IFF2})
	Since $\I\subseteq\J\otimes\fin$, by Lemma~\ref{lemma:Fubini-product-bounded-from-above-by-second-ideal-IFF} we get $\bb(\I,\fin,\fin)\leq\bb(\fin,\fin,\fin)$. By \ref{thm:bNumber-for-all-ideals}(\ref{thm:bNumber-for-all-ideals:FFF-FIF-b}), we have $\bb(\fin,\fin,\fin)=\bb$ and the inequality ``$\geq$'' follows from Theorem \ref{thm:bNumber-for-all-ideals}(\ref{thm:bNumber-for-all-ideals:INEQUALITIES-IFF}).
\end{proof}

The last item of Theorem \ref{thm:bNumber-for-Fubini-products-IFF} may be surprising: $\bb(\J\otimes\{\emptyset\},\fin,\fin)$ does not depend on $\J$. 

\subsubsection{The cardinal $\bb(\I,\I,\I)$.}

We will need three lemmas.

\begin{lemma}
	\label{lemma:Fubini-product-bounded-form-above-by-first-ideal-III}
	Let $\J,\K$ be ideals on $\omega$ 
	(here we allow $\K=\{\emptyset\}$)
	and $\I$ be an ideal on $\omega^2$.
	If $\J\otimes\{\emptyset\}\subseteq\I\subseteq\J\otimes\K$, then $\bb(\I,\I,\I)\leq\bb(\J,\J,\J)$.
\end{lemma}

\begin{proof}
	We will show that $\kappa<\bb(\I,\I,\I)$ implies $\kappa<\bb(\J,\J,\J)$. Let $\kappa<\bb(\I,\I,\I)$ and $\{\langle A^\alpha_n:n\in\omega\rangle:\alpha<\kappa\}\subseteq\cP_\J$.
	
	For each $n\in\omega$ and $\alpha<\kappa$ define $E^\alpha_n=A^\alpha_n\times\omega$. Then $\{\langle E^\alpha_n:n\in\omega\rangle:\alpha<\kappa\}\subseteq\cP_\I$ (as $\J\otimes\{\emptyset\}\subseteq\I$).
	
	Since $\kappa<\bb(\I,\I,\I)$, there is $\langle B_n\rangle\in\cP_\I$ such that $\bigcup_{n\in\omega}(B_n\cap \bigcup_{i\leq n}E^\alpha_i)\in\I$ for all $\alpha<\kappa$. 
	
	Define
	$$V=\left\{k\in\omega:(B_n)_k\in\K \text{ for all $n\in\omega$}\right\}.$$
	We will show that $V\in\J$. Fix $k\in V$. There is $m\in\omega$ such that $k\in A^0_m$. Then 
	$$\left(\bigcup_{n\in\omega}\left(B_n\cap \bigcup_{i\leq n}E^0_i\right)\right)_k
	=
	\left(\bigcup_{n\geq m} B_n\right)_k
	=
	\omega\setminus \left(\bigcup_{n<m} B_n\right)_k\notin\K.$$
Since  $\bigcup_{n\in\omega}(B_n\cap \bigcup_{i\leq n}E^0_i)\in\I$ and $\I\subseteq\J\otimes\K$, we get 
	$$V\subseteq \left\{k\in\omega:  \left(\bigcup_{n\in\omega}\left(B_n\cap \bigcup_{i\leq n}E^0_i\right)\right)_k\notin\K\right\}\in\J.$$
	
	Define $\langle C_n\rangle\in\cP_\J$ by $V\subseteq C_0$ and
	$$k\in C_n \iff n=\min\{i\in\omega: (B_i)_k\notin\K\}$$
	for $k\notin V$ (each $C_n$ belongs to $\J$ as $\langle B_i\rangle\subseteq\I$ and $\I\subseteq\J\otimes\K$).
	
	Fix $\alpha<\kappa$ and denote $X=\bigcup_{n\in\omega}(C_n\cap \bigcup_{i\leq n}A^\alpha_i)$. 
		If we show that $X\in \J$, the proof will be finished. 
	Suppose to the contrary that $X\notin\J$. Then 
	\begin{equation*}
		\begin{split}
	\J\not\ni X\setminus V 
&	=
	\left\{k\in\omega:\exists n\in\omega\ k\in C_n \land k\in\bigcup_{i\leq n}A^\alpha_i \land (B_n)_k\notin \cK\right\}\\
&	=
	\left\{k\in\omega: \exists n\in\omega\ k\in C_n \land   \left(\bigcup_{i\leq n}E^\alpha_i\right)_k=\omega \land (B_n)_k\notin\K\right\}\\
&	\subseteq
	\left\{k\in\omega:\left(\bigcup_{n\in\omega}\left(B_n\cap \bigcup_{i\leq n}E^\alpha_i\right)\right)_k\notin\K\right\},
\end{split}
\end{equation*}
	which contradicts $\bigcup_{n\in\omega}(B_n\cap \bigcup_{i\leq n}E^\alpha_i)\in\I\subseteq\J\otimes\K$. 
\end{proof}

\begin{lemma}
	\label{lemma:Fubini-product-bounded-form-below-by-first-ideal-III}
Let $\J,\K$ be ideals on $\omega$ (here we allow $\K=\{\emptyset\}$).
If $\I=\J\otimes\K$, then  $\bb(\I,\I,\I)\geq\bb(\J,\J,\J)$.
\end{lemma}

\begin{proof}
	We will show that $\kappa<\bb(\J,\J,\J)$ implies $\kappa<\bb(\I,\I,\I)$. 
	
	Let $\kappa<\bb(\J,\J,\J)$ and $\{\langle E^\alpha_n:n\in\omega\rangle:\alpha<\kappa\}\subseteq\I^\omega$.
	
	For each $n\in\omega$ and $\alpha<\kappa$ define $A^\alpha_n=\{i\in\omega:(E^\alpha_n)_i\notin\K\}$. Then $\{\langle A^\alpha_n:n\in\omega\rangle:\alpha<\kappa\}\subseteq\J^\omega$.
	
	Since $\kappa<\bb(\J,\J,\J)$, there is $\langle B_n\rangle\in\cP_\J$ such that $\bigcup_{n\in\omega}(B_n\cap A^\alpha_n)\in\J$ for all $\alpha<\kappa$. 
	
Define $\langle C_n\rangle\in\cP_\I$ by $C_n=B_n\times\omega$ for all $n\in\omega$.

Fix $\alpha<\kappa$ and denote $X=\bigcup_{n\in\omega}(C_n\cap E^\alpha_n)$. 
If we show that $X\in \I$, the proof will be finished. Since $\I=\J\otimes\cK$, it is enough to show that $\{i\in\omega: (X)_i\notin\cK\} \subseteq \bigcup_{n\in\omega}(B_n\cap A^\alpha_n)$.

Take $i\notin \bigcup_{n\in\omega}(B_n\cap A^\alpha_n)$. Since $\langle B_n\rangle$ is a partition, there is $n\in \omega$ with $i\in B_n$.
Then $i\notin A_n^\alpha$, so $(E_n^\alpha)_i\in \cK$.
Moreover, $i\in B_n$ implies $(C_n)_i=\omega$, so 
$(X)_i = (B_n\cap E^\alpha_n)_i = \omega\cap (E_n^\alpha)_i\in \cK$.
\end{proof}

\begin{lemma}
\label{lem:Fubini}
Let $\J,\K$ be ideals on $\omega$, $\J'$ be an ideal on $\omega^2$ and denote $\I=\{\emptyset\}\otimes\K$.
	\begin{enumerate}
		\item If $\{\{0\}\times A: A\in\J\}\subseteq\J'$, then $\bb(\I,\I,\J')\leq\bb(\K,\K,\J)$.
		\item If $\J'\subseteq\{\emptyset\}\otimes\J$, then $\bb(\I,\I,\J')\geq\bb(\K,\K,\J)$.
	\end{enumerate}
\end{lemma}

\begin{proof}
(1)
Let $\{\langle A^\alpha_n: n\in\omega \rangle: \alpha<\bb(\K,\K,\J)\}\subseteq\J^\omega$ be such that for any $\langle B_n: n\in\omega \rangle\in\mathcal{P}_\K$ there is $\alpha<\bb(\K,\K,\J)$ such that $\bigcup_{n\in\omega}(B_n\cap A^\alpha_n)\notin\K$. Consider the family $\{\langle \{0\}\times A^\alpha_n : n\in\omega \rangle: \alpha<\bb(\K,\K,\J)\}\subseteq(\J')^\omega$. Fix any $\langle C_n: n\in\omega \rangle\in\mathcal{P}_\I$ and define $B_n=\{i\in\omega: (0,i)\in C_n\}$ for all $n\in\omega$. Then $\langle B_n: n\in\omega \rangle\in\mathcal{P}_\K$, so there is $\alpha<\bb(\K,\K,\J)$ such that $\bigcup_{n\in\omega}(B_n\cap A^\alpha_n)\notin\K$. We have 
$$\bigcup_{n\in\omega}(C_n\cap (\{0\}\times A^\alpha_n))\supseteq \{0\}\times \left(\bigcup_{n\in\omega}(B_n\cap A^\alpha_n)\right)\notin\I.$$

(2)
Let $\{\langle A^\alpha_n: n\in\omega \rangle: \alpha<\bb(\I,\I,\J')\}\subseteq(\J')^\omega$ be such that for any $\langle B_n: n\in\omega \rangle\in\mathcal{P}_\I$ there is $\alpha<\bb(\I,\I,\J')$ such that $\bigcup_{n\in\omega}(B_n\cap A^\alpha_n)\notin\I$. For each $\alpha<\bb(\I,\I,\J')$ and $n,k\in\omega$ define $E^{\alpha,k}_n=(A^\alpha_n)_{k}$. Consider the family $\{\langle E^{\alpha,k}_n: n\in\omega \rangle: k\in\omega,\alpha<\bb(\I,\I,\J')\}\subseteq\J^\omega$. Fix any $\langle C_n: n\in\omega \rangle\in\mathcal{P}_\K$ and define $B_n=\omega\times C_n$ for all $n\in\omega$. Then $\langle B_n: n\in\omega \rangle\in\mathcal{P}_\I$, so there is $\alpha<\bb(\I,\I,\J')$ such that $\bigcup_{n\in\omega}(B_n\cap A^\alpha_n)\notin\I$. Thus, there is $k\in\omega$ such that $(\bigcup_{n\in\omega}(B_n\cap A^\alpha_n))_{k}\notin\K$. We have 
$$\bigcup_{n\in\omega}(C_n\cap E^{\alpha,k}_n)\supseteq \left(\bigcup_{n\in\omega}(B_n\cap A^\alpha_n)\right)_{k}\notin\K.$$
\end{proof}

\begin{theorem}
		\label{thm:bNumber-for-Fubini-products-III}
	Let $\J,\K$ be ideals on $\omega$.
		\begin{enumerate}
		\item If $\I=\J\otimes\K$ or $\I=\J\otimes\{\emptyset\}$, then $\bb(\I,\I,\I)=\bb(\J,\J,\J)$.\label{thm:bNumber-for-Fubini-products:III}
		\item If $\I=\{\emptyset\}\otimes\K$, then $\bb(\I,\I,\I)=\bb(\K,\K,\K)$.\label{thm:bNumber-for-Fubini-products:III2}
	\end{enumerate}
\end{theorem}

\begin{proof}\ 
	
	(\ref{thm:bNumber-for-Fubini-products:III})
	The inequality ``$\leq$'' follows from Lemma~\ref{lemma:Fubini-product-bounded-form-above-by-first-ideal-III}, and  ``$\geq$'' follows from Lemma~\ref{lemma:Fubini-product-bounded-form-below-by-first-ideal-III}.

	(\ref{thm:bNumber-for-Fubini-products:III2})
	Follows from Lemma \ref{lem:Fubini}.
\end{proof}

\subsubsection{The cardinal $\bb(\I,\I,\fin)$.}

\begin{theorem}
\label{thm:bNumber-for-Fubini-products-IIF}
	Let $\J,\K$ be ideals on $\omega$.
		\begin{enumerate}
		\item If $\I=\J\otimes\K$, then $\bb(\I,\I,\fin)=\cc^+$.\label{thm:bNumber-for-Fubini-products-IIF-JxK}
		\item If $\I=\J\otimes\{\emptyset\}$, then $\bb(\I,\I,\fin)=\bb(\J,\J,\fin)$.\label{thm:bNumber-for-Fubini-products-IIF-Jx0}
		\item If $\I=\{\emptyset\}\otimes\K$, then $\bb(\I,\I,\fin)=\bb(\K,\K,\fin)$.\label{thm:bNumber-for-Fubini-products-IIF-0xK}
	\end{enumerate}
\end{theorem}

\begin{proof}
(\ref{thm:bNumber-for-Fubini-products-IIF-JxK}) Follows from Theorem \ref{thm:bNumber-for-all-ideals}(\ref{thm:bNumber-for-all-ideals:INEQUALITIES-IIF}) since $\J\otimes\K$ is not a weak P-ideal.

(\ref{thm:bNumber-for-Fubini-products-IIF-Jx0}) First we show  $\bb(\I,\I,\fin)\leq\bb(\J,\J,\fin)$.

Let $\{\langle A^\alpha_n: n\in\omega \rangle: \alpha<\bb(\J,\J,\fin)\}\subseteq\fin^\omega$ be such that for any $\langle B_n: n\in\omega \rangle\in\mathcal{P}_\J$ there is $\alpha<\bb(\J,\J,\fin)$ such that $\bigcup_{n\in\omega}(B_n\cap A^\alpha_n)\notin\J$. Consider the family $\{\langle A^\alpha_n\times \{0\}: n\in\omega \rangle: \alpha<\bb(\J,\J,\fin)\}\subseteq\fin(\omega^2)^\omega$. Fix any $\langle C_n: n\in\omega \rangle\in\mathcal{P}_\I$ and define $B_n=\{i\in\omega: (i,0)\in C_n\}$ for all $n\in\omega$. Then $\langle B_n: n\in\omega \rangle\in\mathcal{P}_\J$, so there is $\alpha<\bb(\J,\J,\fin)$ such that $\bigcup_{n\in\omega}(B_n\cap A^\alpha_n)\notin\J$. We have 
$$\bigcup_{n\in\omega}(C_n\cap (A^\alpha_n\times \{0\}))\supseteq \left(\bigcup_{n\in\omega}(B_n\cap A^\alpha_n)\right)\times\{0\}\notin\I.$$

Now we show $\bb(\I,\I,\fin)\geq\bb(\J,\J,\fin)$.

Let $\{\langle A^\alpha_n: n\in\omega \rangle: \alpha<\bb(\I,\I,\fin)\}\subseteq\fin(\omega^2)^\omega$ be such that for any $\langle B_n: n\in\omega \rangle\in\mathcal{P}_\I$ there is $\alpha<\bb(\I,\I,\fin)$ such that $\bigcup_{n\in\omega}(B_n\cap A^\alpha_n)\notin\I$. For each $\alpha<\bb(\I,\I,\fin)$ and $n\in\omega$ define $E^{\alpha}_n=\{i\in\omega: (\exists j\in\omega)((i,j)\in A^\alpha_n)\}$. Consider the family $\{\langle E^{\alpha}_n: n\in\omega \rangle: \alpha<\bb(\I,\I,\fin)\}\subseteq\fin^\omega$. Fix any $\langle C_n: n\in\omega \rangle\in\mathcal{P}_\J$ and define $B_n=C_n\times\omega$ for all $n\in\omega$. Then $\langle B_n: n\in\omega \rangle\in\mathcal{P}_\I$, so there is $\alpha<\bb(\I,\I,\fin)$ such that $\bigcup_{n\in\omega}(B_n\cap A^\alpha_n)\notin\I$. Thus, $\{i\in\omega: (\exists j\in\omega)((i,j)\in \bigcup_{n\in\omega}(B_n\cap A^\alpha_n))\}\notin\J$. We have
$$\bigcup_{n\in\omega}(C_n\cap E^{\alpha}_n)\supseteq \left\{i\in\omega: (\exists j\in\omega)\left((i,j)\in \bigcup_{n\in\omega}\left(B_n\cap A^\alpha_n\right)\right)\right\}\notin\J.$$

(\ref{thm:bNumber-for-Fubini-products-IIF-0xK})
Follows from Lemma \ref{lem:Fubini}.
\end{proof}

\begin{example}
	If $\I=\fin\otimes\{\emptyset\}$, then
	$$\aleph_0=\bb_s(\fin,\I,\I)<\bb  = \bb(\I,\I,\I) = \bb(\I,\fin,\fin) =  \bb(\I,\I,\fin).$$
Indeed, $\bb_s(\fin,\I,\I)=\aleph_0$ by Theorem \ref{thm:bNumber-for-Fubini-products-FII}(\ref{thm:bNumber-for-Fubini-products-FII-notP}). Moreover, $\bb(\I,\fin,\fin)=\bb$ by Theorem \ref{thm:bNumber-for-Fubini-products-IFF}(\ref{thm:bNumber-for-Fubini-products:IFF2}), and $\bb(\I,\I,\fin)=\bb$ by Theorems \ref{thm:bNumber-for-Fubini-products-IIF}(2) and \ref{thm:bNumber-for-all-ideals}(\ref{thm:bNumber-for-all-ideals:FFF-FIF-b}).
	Hence, it is enough to show $\bb\leq \bb(\I,\I,\I)$. But this inequality follows from Theorem~\ref{thm:bNumber-for-Fubini-products-III}(\ref{thm:bNumber-for-Fubini-products:III}).
\end{example}


\section{Nice ideals}
\label{sec:NiceIdeals}

In this section we will compute $\bb(\I,\fin,\fin)$ for ideals $\I$ with the Baire property and $\bb(\I,\I,\fin)$ for coanalytic weak P-ideals. The latter gives an upper bound for $\bb(\I,\I,\I)$ in the case of coanalytic weak P-ideals (by a result of Debs and Saint Raymond, this class contains all $\bf{\Pi^0_4}$ ideals). 


\subsection{Ideals with the Baire property}

In this subsection an interval $[a,b)$ will mean $[a,b)\cap \omega$ i.e.~the set $\{n\in \omega: a\leq n<b\}$.

\begin{theorem}[{Talagrand~\cite[Th\'{e}or\`{e}me 21]{MR579439} (see also \cite[Theorem 4.1.2]{MR1350295})}]
\label{thm:talagrand-characterization}
An ideal $\I$ on $\omega$ has the Baire property if and only if there is an increasing sequence $n_1<n_2<\dots$ such that 
if there are infinitely many $k$ with $[n_k,n_{k+1})\subseteq A$, then $A\notin \I$.
\end{theorem}

\begin{corollary}
\label{cor:talagrand-for-functions}
Let $\I$ be an ideal with the Baire property.
For every $g\in\omega^\omega$ there is $h\in\omega^\omega$ such that 
\begin{enumerate}
\item $h$ is strictly increasing,
\item $g\leq h$,
\item if $A\in\I$, then the set $\{n : [h(n),h(n+1))\subseteq A\}$ is finite.
\end{enumerate}
\end{corollary}

\begin{proof}
By Theorem~\ref{thm:talagrand-characterization} there is a sequence $n_0<n_1<\dots$ such that 
$\{k: [n_k,n_{k+1})\subseteq A\}$ is finite  for every $A\in\I$. 

It is enough to define $h$ such that 
$h$ is increasing, 
$g\leq h$
and
$[h(n),h(n+1))$ contains at least one interval $[n_k,n_{k+1})$ for every $n\in\omega$.
We can define $h$ inductively in the following manner.
Let $h(0)=g(0)$.
Suppose that $h(i)$ has been defined for $i\leq n$.
Let $k_0=\min\{k: h(n)\leq n_k\}$.
We put $h(n+1)=\max\{n_{k_0+1},h(n)+1,g(n+1)\}$.
\end{proof}

\begin{theorem}
\label{thm:bNumber-IFF-for-ideals-with-BP}
If an ideal $\I$ has the Baire property, then $\bb(\I,\fin,\fin)=\bb$.
\end{theorem}

\begin{proof}
By Theorem~\ref{thm:bNumber-for-all-ideals}(\ref{thm:bNumber-for-all-ideals:INEQUALITIES-IFF})
we only have to show that 
$\bnumber(\I,\fin,\fin)\leq \bnumber$. We will show that 
$\kappa<\bnumber(\I,\fin,\fin)$ implies $\kappa<\bnumber$.

Let $\kappa<\bnumber(\I,\fin,\fin)$ and $g_\alpha\in\omega^\omega$ for $\alpha<\kappa$.

By Corollary~\ref{cor:talagrand-for-functions} we can assume that 
$g_\alpha$ are strictly increasing
and
$$\bigcup_{n\in A}[g_\alpha(n),g_\alpha(n+1))\notin\I$$ 
for every infinite $A\subseteq\omega$ and $\alpha<\kappa$.

For $\alpha<\kappa$ we define
$A^\alpha_0 = [0,g_\alpha(1))$ and 
$A^\alpha_n = [g_\alpha(n),g_\alpha(n+1))$ for $n\geq 1$.

Since $\{\langle A^\alpha_n:n\in \omega\rangle: \alpha<\kappa\}\subseteq\cP_{\fin}$
and $\kappa<\bnumber(\I,\fin,\fin)$, there is a partition $\langle B_n: n\in\omega\rangle\in \cP_\fin$ such that $\bigcup_{n\in\omega}(B_n\cap\bigcup_{i\leq n}A^\alpha_i)\in \I$ for every $\alpha<\kappa$.

We define $g\in\omega^\omega$ by
$g(n) = \max(\bigcup_{i\leq n} B_i)$.
If we show that $g_\alpha\leq^* g$ for every $\alpha<\kappa$, the proof will be finished.

Let $\alpha<\kappa$. 
Since
 $B = \bigcup_{n\in\omega}(B_n\cap\bigcup_{i\leq n}A^\alpha_i)\in \I$,
the set $A = \{n: [g_\alpha(n),g_\alpha(n+1))\subseteq B\}$ is finite.

Now we show that $g_\alpha(n)\leq g(n)$ for every $n\in \omega\setminus A$ (i.e. $g_\alpha\leq^* g$).
Let $n\in \omega\setminus A$.  
Since $[g_\alpha(n),g_\alpha(n+1))\cap (\omega\setminus B)\neq \emptyset$, there is 
$m\in [g_\alpha(n),g_\alpha(n+1))\cap (\omega\setminus B)$. 
Then $g_\alpha(n)\leq m$. 
On the other hand,
$m\notin B = \bigcup_{n\in\omega}(B_n\cap\bigcup_{i\leq n}A^\alpha_i)$
and 
$m\in [g_\alpha(n),g_\alpha(n+1)) = A^\alpha_n$.
Thus $m\notin B_i$ for every $i\geq n$,
so
$m\in B_i$ for some $i<n$.
Hence
$g(n) = \max(\bigcup_{i\leq n} B_i) \geq m \geq g_\alpha(n)$.
\end{proof}

\begin{corollary}
\label{thm:bNumber-for-P-ideals-with-BP}
If $\I$ is a P-ideal with the Baire property, then
$$\aleph_1\leq \bb_s(\fin,\I,\I)\leq \bb=\bb(\I,\fin,\fin)=\bb(\I,\I,\I)=\bb(\I,\I,\fin).$$
\end{corollary}

\begin{proof}
Follows from Theorems~\ref{thm:bNumber-for-P-ideals} and \ref{thm:bNumber-IFF-for-ideals-with-BP}.
\end{proof}

For 
$f:\omega\to[0,\infty)$ satisfying $\sum_{n\in\omega}f(n)=\infty$ we define the \emph{summable ideal} by 
$$\I_f=\left\{A\subset\omega:\sum_{n\in A}f(n)<\infty\right\}$$
(for instance, $\I_{1/n} = \{A\subseteq\omega: \sum_{n\in A}1/(n+1)<\infty\}$ is the summable ideal), 
and for 
$f:\omega\to[0,\infty)$ satisfying $\sum_{n\in\omega}f(n)=\infty$ and $\lim_{n\to\infty}f(n)/(\sum_{i\leq n}f(i))=0$,
we define the \emph{Erd\H{o}s-Ulam ideals} by
$$EU_f=\left\{A\subset\omega: \limsup_{n\to\infty}\frac{\sum_{i\in A\cap n}f(i)}{\sum_{i\in n}f(i)}=0\right\}$$
(for instance, the ideal $\I_d = \{A:\lim_n|A\cap n|/n=0\}$
of all sets of asymptotic density zero is an Erd\H{o}s-Ulam ideal).

\begin{corollary}
\label{cor:bNumber-for-summable-and-EU-ideals}
If $\I$ is a tall summable ideal or a tall Erd\H{o}s-Ulam ideal, then 
$$\add(\cN) = \bb_s(\fin,\I,\I) \leq \bb=\bb(\I,\fin,\fin)=\bb(\I,\I,\I)=\bb(\I,\I,\fin),$$
where $\add(\cN)$ is the smallest size of a family of Lebesgue null sets with non-null union.
\end{corollary}

\begin{proof}
It was proved by Hern\'{a}ndez and Hru\v{s}\'{a}k~\cite[Theorem~2.2]{MR2319159}
that $\adds(\I)=\add(\cN)$ in these cases.
Since $\add(\cN)\leq \bb$ (see e.g.~\cite[p.~35]{MR1350295}), we get 
$\bb_s(\fin,\I,\I) = \add(\cN)$. 
	Moreover, since $\I$ is a P-ideal with the Baire property (see e.g.~\cite[Examples~1.2.3(c,d)]{MR1711328}), 
Corollary~\ref{thm:bNumber-for-P-ideals-with-BP} finishes the proof.
\end{proof}

\begin{corollary}
	\label{cor:finXfinXfinXdots}
	Let $n\in\omega\setminus\{0,1\}$.
	If $\I=\fin^n$, then
	$$\aleph_0 = \bb_s(\fin,\I,\I)< \bb  = \bb(\I,\fin,\fin) = \bb(\I,\I,\I)\leq \dd\leq \continuum< \bb(\I,\I,\fin) = \continuum^+.$$
\end{corollary}

\begin{proof}
	Since $\I$ in not a P-ideal, we get $\bb_s(\fin,\I,\I) = \aleph_0$.
	Since $\I$ has the Baire property (see e.g.~\cite{MR2777744}), we can apply 
	Theorem~\ref{thm:bNumber-IFF-for-ideals-with-BP} to obtain
	$\bb(\I,\fin,\fin) = \bb$.
		Since $\I = \fin^n = \fin\otimes\fin^{n-1}$, we can apply 
	Theorem~\ref{thm:bNumber-for-Fubini-products-III}(\ref{thm:bNumber-for-Fubini-products:III}) to obtain 
	$\bb(\I,\I,\I) = \bb(\fin,\fin,\fin) = \bb$.
	Since $\I$ is not a weak P-ideal, we can apply Proposition~\ref{thm:bNumber-for-all-ideals}(\ref{thm:bNumber-for-all-ideals:IIF-nonWeakP}) to obtain $\bb(\I,\I,\fin) = \continuum^+$.
\end{proof}


\subsection{$\omega$-diagonalizable ideals}
\label{sec:omega-diag}

We need to recall the following technical notion. We will apply it in the next subsection to the cases of coanalytic weak P-ideals and $\bf{\Pi^0_4}$ ideals.

Let $\I$ be an ideal on $\omega$.
A family $\cZ\subseteq [\omega]^{<\omega}\setminus \{\emptyset\}$ is \emph{$\I^*$-universal} if for each $F\in \I^*$ there is $Z\in \cZ$ with $Z\subseteq F$. 
We say that $\I$ is \emph{$\omega$-diagonalizable by $\mathcal{I}^*$-universal sets} if there exists a family $\{\cZ_k:k\in\omega\}$ of $\I^*$-universal families such that for each $F\in \I^*$ there is $k\in\omega$ such that $Z\cap F\neq \emptyset$ for every $Z\in \cZ_k$. 

\begin{theorem}
\label{thm:bNumber-for-omega-diag}
If an ideal $\I$ is $\omega$-diagonalizable by $\mathcal{I}^*$-universal sets, then 
$$\aleph_1\leq \bb(\I,\I,\I)\leq \bb = \bb(\I,\fin,\fin)=\bb(\I,\I,\fin).$$
\end{theorem}

\begin{proof}
By Theorem~\ref{thm:bNumber-for-all-ideals}, we only need to show $\bb(\I,\I,\fin)\leq\bb$.

Let $\{f_\alpha:\alpha<\bb\}$ be an unbounded family on $\omega^\omega$. Without loss of generality we can assume that each $f_\alpha$ is strictly increasing.

For each $n\in\omega$ and $\alpha<\bb$ define $F^\alpha_n=\{i\in\omega:i\leq f_\alpha(n)\}$. We will show that $\{\langle F^\alpha_n:n\in\omega\rangle: \alpha<\bb\}\subseteq\fin^\omega$ witnesses $\bb(\I,\I,\fin)\leq\bb$.

Fix $\langle B_n:n\in\omega\rangle\in\cP_\I$. Let $\{\cZ_k:k\in\omega\}$ be a family of $\I^*$-universal sets $\omega$-diagonalizing $\I$. For each $k$, let $\cZ_k=\{Z_k^n:n\in \omega\}$.

For each $k,n\in\omega$ find $i(k,n)\in\omega$ such that $Z_k^{i(k,n)}\cap(B_0\cup\ldots\cup B_n)=\emptyset$ and define $g_k\in\omega^\omega$ by $g_k(n)=\max Z_k^{i(k,n)}$. Let $g\in\omega^\omega$ be given by $g(i)=\max\{g_1(i),\ldots,g_i(i)\}$. 

There is $\alpha<\bb$ such that $f_\alpha(i)>g(i)$ for infinitely many $i$. If we show that for each $k\in\omega$ there is $i\in\omega$ with $Z_k^i\subseteq \bigcup_{j\in\omega}(B_j\cap F^\alpha_j)$, then, by $\omega$-diagonalizability of $\I$, it will follow that $\bigcup_{j\in\omega}(B_j\cap F^\alpha_j)\notin\I$, and the proof will be finished.

Fix $k\in\omega$. There is $n>k$ such that $f_\alpha(n)>g(n)\geq g_k(n)=\max Z_k^{i(k,n)}$. Thus, $Z_k^{i(k,n)}\subseteq F^\alpha_n\subseteq F^\alpha_{n+1}\subseteq\ldots$ and $Z_k^{i(k,n)}\subseteq B_{n+1}\cup B_{n+2}\cup\ldots$. Hence, $Z_k^{i(k,n)}\subseteq \bigcup_{j>n}(B_j\cap F^\alpha_j)\subseteq\bigcup_{j\in\omega}(B_j\cap F^\alpha_j)$ and we are done.
\end{proof}


\subsection{Coanalytic weak P-ideals and $\bf{\Pi^0_4}$ ideals}
\label{sec:coanalytic-ideals}

\begin{theorem}
\label{thm:bNumber-coanalytic-weak-P-ideals}
If $\I$ is a coanalytic weak P-ideal, then 
$$\aleph_1\leq \bb(\I,\I,\I)\leq \bb = \bb(\I,\fin,\fin)=\bb(\I,\I,\fin).$$
\end{theorem}

\begin{proof}
Consider the game $G\left(\I\right)$, defined by Laflamme (see \cite{Laflamme}) as follows: Player I in his $n$'th move plays an element $C_{n}\in\I$, and then Player II responses with any $F_{n}\in \left[\omega\right]^{<\omega}$ such that $F_{n}\cap C_{n}=\emptyset$. Player I wins if $\bigcup_{n\in \omega} F_{n}\in \I$. Otherwise, Player II wins.

By \cite[Theorem 5.1]{MR3624783}, $G\left(\I\right)$ is determined for coanalytic ideals (see also \cite[Theorem 1.6]{Sabok}). Moreover, by \cite[Theorem 2.16]{Laflamme}, Player I has a winning strategy in $G\left(\I\right)$ if and only if $\I$ is not a weak P-ideal. Thus, in our case Player II has a winning strategy. Again by \cite[Theorem 2.16]{Laflamme}, this is in turn equivalent to $\I$ being $\omega$-diagonalizable by $\mathcal{I}^*$-universal sets. Using Theorem~\ref{thm:bNumber-for-omega-diag} we get $\aleph_1\leq \bb(\I,\I,\I)\leq \bb = \bb(\I,\fin,\fin)=\bb(\I,\I,\fin)$ and we are done.
\end{proof}

\begin{corollary}
\label{cor:bNumber-for-Pi^0_4}
If $\I$ is a $\bf{\Pi^0_4}$ ideal, then 
$$\aleph_1\leq \bb(\I,\I,\I)\leq \bb = \bb(\I,\fin,\fin)=\bb(\I,\I,\fin).$$
\end{corollary}

\begin{proof}
By \cite[Theorems 7.5 and 9.1]{Debs}, each $\bf{\Pi^0_4}$ ideal is a weak P-ideal. Thus, the corollary follows from Theorem~\ref{thm:bNumber-coanalytic-weak-P-ideals}.
\end{proof}

\begin{example}
The \emph{eventually different ideal} is defined by
 $$\ED=\{A\subseteq\omega\times\omega: \exists m,n\in\omega\,\forall k\geq n\,(|\{i: (k,i)\in A\}|\leq m)\}.$$
If $\I =\ED$, then
	$$\aleph_0=\bb_s(\fin,\I,\I)<\bb(\I,\I,\I)\leq\bb =  \bb(\I,\fin,\fin) =  \bb(\I,\I,\fin).$$
Indeed, since $\I$ is not a P-ideal, $\bb_s(\fin,\I,\I)=\aleph_0$ (see Lemma~\ref{lem:SbNumber-for-all-ideals}(\ref{lem:SbNumber-for-all-ideals:FII-nonPideal})).
Since $\I$ is a $\bf{\Sigma^0_2}$ ideal,
we can use Theorem~\ref{cor:bNumber-for-Pi^0_4} to obtain the remaining (in)equality.
(Note that 
since $\fin\otimes\{\emptyset\}\subseteq \ED \subseteq \fin\otimes\fin$, we could also use Lemma~\ref{lemma:Fubini-product-bounded-form-above-by-first-ideal-III} to obtain the inequality $\bb(\I,\I,\I)\leq\bb$.)
\end{example}


\section{Ideals with $\bb(\I,\I,\I)=\aleph_1$}
\label{sec:bb-equals-aleph-one}

In this section we will show that $\bb(\I,\I,\I)=\aleph_1$ for some ideals (even $\bf{\Sigma^0_2}$ ideals). All results of this section follow from the next lemma.

\begin{lemma}
	\label{prawieAD}
	Let $\cA\subseteq [\omega]^\omega$ be an uncountable family with the property that $A\setminus (A_0\cup\ldots\cup A_n)\notin\fin$ for all $n\in\omega$ and $A,A_0,\ldots A_n\in\cA$ with $A\neq A_i$ for all $i\leq n$. If $\I$ is an ideal generated by $\cA$, then $\bb(\I,\I,\I)=\aleph_1$.
\end{lemma}

\begin{proof}
	Let $\left\{A^\alpha_n\in \cA: n\in\omega,\alpha<\aleph_1\right\}$ be such that $A^\alpha_n\neq A^\beta_m$ whenever $(\alpha,n)\neq(\beta,m)$. We will show that 
	$\{\langle \bigcup_{i\leq n} A^\alpha_i: n\in\omega\rangle : \alpha<\aleph_1\}\subseteq\I^\omega$ 
	witnesses $\bb(\I,\I,\I)=\aleph_1$.
	
	Let $\langle B_n:n\in\omega\rangle\in\cP_\I$.
	Then there is $\alpha<\aleph_1$ such that $A^\alpha_n\setminus \bigcup_{i<m}B_i\notin\fin$ for all $n,m\in\omega$. 
	Indeed, 
	for each $m$ there are $C^m_0,\ldots,C^m_{l_m}\in\cA$ such that $\bigcup_{i<m}B_i\subseteq^* C^m_0\cup\ldots\cup C^m_{l_m}$.
	Since $\cC=\{C^m_{i}: m\in\omega, i\leq l_m\}$ is countable, there is $\alpha<\aleph_1$ with $\{A^\alpha_n:n\in\omega\}\cap \cC=\emptyset$. Then
	$$A^\alpha_n\setminus \bigcup_{i<m}B_i\supseteq^* A^\alpha_n\setminus(C^m_0\cup\ldots\cup C^m_{l_m})\notin\fin$$ 
	for each $n,m\in\omega$.

	Let $B=\bigcup_{n\in\omega}(B_n\cap \bigcup_{i\leq n}A^\alpha_i)$. 
	If we show that $B\notin \I$, the proof will be finished.
	Suppose to the contrary that $B=\bigcup_{n\in\omega}(B_n\cap \bigcup_{i\leq n}A^\alpha_i)\in\I$. Then there are $D_0,\ldots,D_k\in\cA$ with $B\subseteq^* D_0\cup\ldots\cup D_k$. Find $j\in\omega$ such that $A^\alpha_j\neq D_i$ for all $i=0,\ldots,k$. 
	
	We have $A^\alpha_j\setminus\bigcup_{n<j}B_n\notin\fin$ and 
	$$A^\alpha_j\setminus\bigcup_{n<j}B_n=A^\alpha_j\cap\bigcup_{n\geq j}B_n\subseteq \bigcup_{n\geq j}\left(B_n\cap\bigcup_{i\leq n}A^\alpha_i\right)\subseteq\bigcup_{n\in\omega}\left(B_n\cap\bigcup_{i\leq n}A^\alpha_i\right)=B.$$
	
	On the other hand, $(A^\alpha_j\setminus\bigcup_{n<j}B_n)\setminus\bigcup_{n\leq k}D_n=A^\alpha_j\setminus(\bigcup_{n<j}B_n\cup\bigcup_{n\leq k}D_n)\supseteq^* A^\alpha_j\setminus(\bigcup_{n\leq l_j}C^j_{n}\cup\bigcup_{n\leq k}D_n)\notin\fin$. This contradicts $B\subseteq^* D_0\cup\ldots\cup D_k$ and finishes the proof.
\end{proof}

\begin{theorem}
	\label{AD}
	Let $\cA\subseteq [\omega]^\omega$ be an uncountable almost disjoint family.
	If $\I$ is an ideal generated by $\cA$, then $\bb(\I,\I,\I)=\aleph_1$.
\end{theorem}

\begin{proof}
	This follows directly from Lemma \ref{prawieAD} -- it suffices to observe that for any almost disjoint family $\cA$, if $A,A_0,\ldots A_n\in\cA$ are such that $A\neq A_i$ for all $i\leq n$, then $A\cap(A_0\cup\ldots\cup A_n)$ is finite, so $A\setminus (A_0\cup\ldots\cup A_n)$ is infinite. 
\end{proof}

Let $2^{<\omega}$ be the set of all finite sequences of zeros and ones.
Let $\I_b$ denote the ideal on $2^{<\omega}$ generated by all branches i.e.~sets of the form $B_x=\{s\in 2^{<\omega}:s\subseteq x\}$ for $x\in 2^\omega$. It is easy to see that the ideal $\I_b$ is $\bf{\Sigma^0_2}$ and not tall. 

\begin{corollary}
\label{cor:bNumber-Ib=aleph1}
	$\bb(\I_b,\I_b,\I_b)=\aleph_1$.
\end{corollary}

\begin{proof}
	It follows from Theorem \ref{AD}, as the family $\{B_x:x\in 2^\omega\}$ is almost disjoint of size $\cc$. 
\end{proof}

Let $\Omega$ be the set of all clopen subsets of the Cantor space $2^\omega$ having Lebesgue measure $1/2$ (note that $\Omega$ is countable).
Let $\mathcal{S}$ denote the Solecki's ideal on $\Omega$ i.e.~the ideal generated by sets of the form $G_x=\{A\in\Omega:x\in A\}$ for all $x\in 2^\omega$. $\mathcal{S}$ is a tall $\bf{\Sigma^0_{2}}$ ideal (see e.g.~\cite{MR2777744}).

\begin{theorem}
	$\bb(\mathcal{S},\mathcal{S},\mathcal{S})=\aleph_1$.
\end{theorem}

\begin{proof}
	It follows from Lemma \ref{prawieAD}, as the family $\{G_x:x\in 2^\omega\}$ has cardinality $\cc$ and for any $x,x_0,\ldots,x_n\in 2^\omega$ with $x\neq x_i$ for all $i\leq n$, we have $G_x\setminus (G_{x_0}\cup\ldots\cup G_{x_n})=\{A\in\Omega: x\in A \land \forall i\leq n \, (x_i\notin A)\}$ is infinite.
\end{proof}


\section{Consistency results}
\label{sec:ConsistencyResults}

In this section we apply some known consistency results to the case of the considered cardinals. The main aim of this section is to show that consistently the values of considered cardinals can be pairwise distinct. We start with two results about $\bb_s(\fin,\I,\I)$.

\begin{proposition}
	It is consistent that there is a P-ideal $\I$ with 
	$$\bb_s(\fin,\I,\I) < \bb(\I,\I,\I) = \bb(\I,\fin,\fin)=\bb(\I,\I,\fin).$$
\end{proposition}

\begin{proof}
It follows from Corollary~\ref{cor:bNumber-for-summable-and-EU-ideals}
and the fact that it is consistent that $\add(\cN)<\bb$ (see e.g.~\cite{MR2768685}).
\end{proof}

\begin{theorem}[Essentially Louveau]
	Under Martin's axiom, there is a maximal ideal $\I$ with $\bb_s(\fin,\I,\I) = \kappa$ for all regular $\aleph_1\leq \kappa\leq \continuum$.
\end{theorem}

\begin{proof}
Louveau~\cite[Théorèmes 3.9 and 3.12]{MR411971} proved that there is a maximal ideal $\I$ with $\adds(\I)=\kappa$ (see also \cite[p.~2647]{MR1686797}).
On the other hand, it is well known that $\bb=\continuum$ under Martin's axiom.
Thus, $\bb_s(\fin,\I,\I) = \kappa$.
\end{proof}

Now we want to reformulate Canjar's results from \cite{Canjar2}, \cite{Canjar} and \cite{CanjarPhD}. He studied $\dnumber(\geq_\I\cap (\cD_\fin\times\cD_\fin))$ and $\dnumber(\geq_\I\cap (\cD_\I\times\cD_\I))$ in the case of maximal ideals $\I$. It is not difficult to see that for a maximal ideal $\I$ we have 
$\dnumber(\geq_\I\cap (\cD_\fin\times\cD_\fin)) = \bnumber(\geq_\I\cap (\cD_\fin\times\cD_\fin))$ and $\dnumber(\geq_\I\cap (\cD_\I\times\cD_\I)) = \bnumber(\geq_\I\cap (\cD_\I\times\cD_\I))$.
On the other hand, by Theorem~\ref{thm:Canjar-bNumber-characterization-by-Staniszewski-bNumber}(\ref{thm:Canjar-bNumber-characterization-by-Staniszewski-bNumber-nonstrict-inequalities}), 
$\bb(\I,\fin,\fin)=\bnumber(\geq_\I\cap (\cD_\fin\times\cD_\fin))$ and $\bb(\I,\I,\I)=\bnumber(\geq_\I\cap (\cD_\I\times\cD_\I))$. Thus, we have the following two results.

\begin{theorem}[Canjar~\cite{Canjar2} and \cite{CanjarPhD}]\ 
	The following is true in the model obtained by adding $\lambda$ Cohen reals to a model of GCH.
	\begin{enumerate}
		\item There exists an ideal $\I$ with $\bb(\I,\fin,\fin)=\kappa$  for all regular cardinals $\aleph_1\leq \kappa< \lambda$.
		\item There exist $2^\kappa$ ideals $\I$ with  $\bb(\I,\I,\I)=\kappa$ for all regular cardinals $\aleph_1\leq \kappa< \lambda$.
	\end{enumerate}
\end{theorem} 

\begin{theorem}[Essentially Canjar~\cite{Canjar}]\
\label{thm:bNumber-equals-cof-d}
\begin{enumerate}
\item There is an ideal $\I$ such that $\bb(\I,\fin,\fin)=\text{cf}(\dd)$.
\item If $\dd=\cc$, then there is a P-ideal $\I$ such that $\bb(\I,\fin,\fin)=\cf(\dd)$.
\end{enumerate}
\end{theorem} 

Recall that consistency of $\bb<\cf(\dd)$ and $\bb<\cf(\dd)\leq\dd=\cc$ follows for instance from \cite[Theorem 2.5]{MR2768685}. 

\begin{corollary}\ 
\label{cor:CON-b-less-than-bIFF--for-P-ideals}
\label{thm:CON-b-less-than-bIFF}
\begin{enumerate}
\item If $\bb<\cf(\dd)$, then there is an ideal $\I$ such that $\bb<\bb(\I,\fin,\fin)=\text{cf}(\dd)$.
\item If $\bb<\cf(\dd)\leq\dd=\cc$, then there is a P-ideal $\I$ such that
$$\bb<\bb(\I,\fin,\fin)=\bb(\I,\I,\I)=\bb(\I,\I,\fin)=\text{cf}(\dd).$$
\end{enumerate}
\end{corollary} 

\begin{proof}
Follows directly from Theorems~\ref{thm:bNumber-equals-cof-d} and \ref{thm:bNumber-for-P-ideals}.
\end{proof}

Next three results will establish that consistently the values of $\bb(\I,\I,\I)$, $\bb(\I,\fin,\fin)$ and $\bb(\I,\I,\fin)$ can be pairwise distinct

\begin{theorem}
\label{cor:CON-b-equals-bIII-less-than-bIFF-for-non-weak-P-ideals}
If $\bb<\bb(\J,\fin,\fin)$ for some ideal $\J$ (e.g. if $\bb<\cf(\dd)$), then there is an ideal $\I$, which is not a weak P-ideal, such that $\bb(\I,\I,\I)=\bb<\bb(\I,\fin,\fin)$.
\end{theorem}

\begin{proof}
Consider the ideal $\I=\fin\otimes\J$. Obviously, $\I$ is not a weak P-ideal as $\fin\otimes\fin\subseteq\I$. By Theorems \ref{thm:bNumber-for-Fubini-products-III}(\ref{thm:bNumber-for-Fubini-products:III}), \ref{thm:bNumber-for-Fubini-products-IFF}(\ref{thm:bNumber-for-Fubini-products:IFF}) and \ref{thm:bNumber-for-all-ideals}(\ref{thm:bNumber-for-all-ideals:FFF-FIF-b}) we have $\bb(\I,\I,\I)=\bb(\fin,\fin,\fin)=\bb$ and $\bb(\I,\fin,\fin)=\bb(\J,\fin,\fin)>\bb$. 
\end{proof}

\begin{theorem}
\label{cor:CON-b-equals-bIFF-less-than-bIII-for-non-weak-P-ideals}
If $\bb<\bb(\J,\J,\J)$ for some ideal $\J$ (e.g. if $\bb<\cf(\dd)\leq\dd=\cc$), then there is an ideal $\I$, which is not a weak P-ideal, such that $\bb(\I,\fin,\fin)=\bb<\bb(\I,\I,\I)$.
\end{theorem}

\begin{proof}
Consider the ideal $\I=\J\otimes\fin$. Obviously, $\I$ is not a weak P-ideal as $\fin\otimes\fin\subseteq\I$. By Theorems \ref{thm:bNumber-for-Fubini-products-III}(\ref{thm:bNumber-for-Fubini-products:III}), \ref{thm:bNumber-for-Fubini-products-IFF}(\ref{thm:bNumber-for-Fubini-products:IFF}) and \ref{thm:bNumber-for-all-ideals}(\ref{thm:bNumber-for-all-ideals:FFF-FIF-b}) we have $\bb(\I,\I,\I)=\bb(\J,\J,\J)>\bb$ and $\bb(\I,\fin,\fin)=\bb(\fin,\fin,\fin)=\bb$. 
\end{proof}

\begin{theorem}
\label{cor:CON-b-equals-bIFF-equals-bIII-less-than-bIIF-for-weak-P-ideals}
If $\bb<\bb(\J,\fin,\fin)$ for some ideal $\J$ (e.g. if $\bb<\cf(\dd)$), then there is a weak P-ideal $\I$ such that
$$\bb(\I,\I,\I)\leq\bb=\bb(\I,\fin,\fin)<\bb(\I,\I,\fin).$$
\end{theorem}

\begin{proof}
Consider the ideal $\I=(\fin\otimes\fin)\cap(\J\otimes\{\emptyset\})$. 

First we show that $\I$ is a weak P-ideal (this fact is also shown in \cite[Lemma 2.3]{MR3784399}, however we prove it here for the sake of completeness). Fix a partition $\langle X_n\rangle\in\mathcal{P}_\I$. Define by induction two sequences $\langle m_n\rangle,\langle k_n\rangle\in \omega^\omega$ such that for each $n\in\omega$ we have $(n,m_n)\in X_{k_n}$ and
$$m_n=
\begin{cases}
\min\{m\in\omega:m\notin(\bigcup\{X_{k_i}:i<n\})_n\} & \text{if $\omega\not\subseteq(\bigcup_{i<n}X_{k_i})_n$},\\
\min\{m\in\omega:m\in(\bigcup\{X_{k}:|(X_k)_n|=\omega\})_n\} & \text{otherwise.}\\
\end{cases}$$
Then $Y=\{(n,m_n):n\in\omega\}\notin\I$ as $\{n\in\omega: (Y)_n\neq\emptyset\}=\omega\notin\J$. Moreover, $Y\cap X_n$ is finite for all $n$ (otherwise we would have $|(X_n)_k|=\omega$ for infinitely many $k\in\omega$). Thus $\I$ is a weak P-ideal.

As $\I\subseteq\fin\otimes\fin$ and $\fin\otimes\fin$ is meager, $\I$ is meager as well. Thus, $\I$ has the Baire property and $\bb(\I,\fin,\fin)=\bb$ by Theorem \ref{thm:bNumber-IFF-for-ideals-with-BP}.

By Lemma \ref{lemma:Fubini-product-bounded-form-above-by-first-ideal-III} we have $\bb(\I,\I,\I)\leq\bb(\fin,\fin,\fin)=\bb$, as $\fin\otimes\{\emptyset\}\subseteq\I\subseteq\fin\otimes\fin$.

To finish the proof we need to show that $\bb(\I,\I,\fin)\geq\bb(\J,\fin,\fin)$. Let $\kappa<\bb(\J,\fin,\fin)$ and $\{\langle E^\alpha_n: n\in\omega\rangle :\alpha<\kappa\}\subseteq \fin(\omega^2)^{\omega}$. We will show that this family does not witness $\bb(\I,\I,\fin)\leq \kappa$. 

For each $n\in\omega$ and $\alpha<\kappa$ define 
$$A^\alpha_n=\{k\in\omega: (E^\alpha_n)_k \neq \emptyset\}.$$
Since $\{\langle A^\alpha_n: n\in\omega\rangle :\alpha<\kappa\}\subseteq\fin^\omega$ 
and
$\kappa<\bb(\J,\fin,\fin)$, there is $\langle B_n\rangle\in\cP_\fin$ such that $\bigcup_{n\in\omega}\left(B_n\cap A^\alpha_n\right)\in\J$ for all $\alpha<\kappa$. 

Define $C_n=B_n\times\omega$ for all $n\in\omega$. Then $\langle C_n\rangle\in\cP_\I$. 

Fix $\alpha<\kappa$ and denote $X=\bigcup_{n\in\omega}\left(C_n\cap E^\alpha_n\right)$. 
If we show that $X\in \I$, the proof will be finished.

Observe that for each $k\in\omega$ there is $j\in\omega$ with $k\in B_j$. Then $(X)_k=(E^\alpha_j)_k$, so $(X)_k$ is finite. Thus, $X\in\{\emptyset\}\otimes\fin\subseteq\fin\otimes\fin$.

Moreover, 
\begin{equation*}
	\begin{split}
\left\{k\in\omega: (X)_k\neq\emptyset\right\}&
=
\left\{k\in\omega: \exists n\in\omega\,\exists i\in\omega\, (k,i)\in C_n\cap E^\alpha_n\right\}\\
&
=
\bigcup_{n\in\omega} \left\{k\in\omega: \exists i\in\omega\, (i \in (C_n\cap E^\alpha_n)_k)\right\}\\
&\subseteq 
\bigcup_{n\in\omega} 
\left(\left\{k\in\omega: \exists i\in\omega(i\in (C_n)_k)\right\}\cap 
\left\{k\in\omega: \exists i\in\omega (i \in (E^\alpha_n)_k)\right\}\right)\\
&=
\bigcup_{n\in\omega}\left(B_n\cap A^\alpha_n\right)\in\J.
\end{split}
\end{equation*}
Thus, $X\in\J\otimes\{\emptyset\}$ and we can conclude that $X\in\I$.  
\end{proof}

We were not able to compute the exact value of $\bb(\I,\I,\I)$ in Theorem \ref{cor:CON-b-equals-bIFF-equals-bIII-less-than-bIIF-for-weak-P-ideals}. Thus, we have the following open question.

\begin{question}
	Is it consistent that there is a weak P-ideal $\I$ such that	$\bb(\I,\I,\I)<\bb(\I,\fin,\fin)<\bb(\I,\I,\fin)$?
\end{question}

There is another closely related open question.

\begin{question}
	Is it consistent that there is a weak P-ideal $\I$ such that	$\bb(\I,\fin,\fin)<\bb(\I,\I,\I)<\bb(\I,\I,\fin)$?
\end{question}


\section{Questions and remarks about $\bb(\I,\I,\I)$}

In this section we collect some remarks indicating difficulties in obtaining any general results concerning $\bb(\I,\I,\I)$ for Borel ideals.

\begin{remark}
Observe that $\I\subseteq\J$ does not give us any information about the relation between $\bb(\I,\I,\I)$ and $\bb(\J,\J,\J)$. Indeed, let $\I$ be an ideal generated by an uncountable family $\cF\subseteq\omega^\omega$ of pairwise almost disjoint graphs (i.e., $\{n\in\omega:f(n)=g(n)\}\in\fin$ for any $f,g\in\cF$, $f\neq g$). Then $\fin(\omega^2)\subseteq\I\subseteq\fin\otimes\fin$, however $\bb(\fin(\omega^2),\fin(\omega^2),\fin(\omega^2))=\bb(\fin\otimes\fin,\fin\otimes\fin,\fin\otimes\fin)=\bb$ (by Corollary~\ref{cor:finXfinXfinXdots}) and $\bb(\I,\I,\I)=\aleph_1$ (by Theorem \ref{AD}).
\end{remark}

\begin{remark}
We have non-tall ideals with $\bb(\I,\I,\I)=\aleph_1$ (e.g. the branching ideal $\I_b$) as well as non-tall ideals with $\bb(\I,\I,\I)=\bb$ (this is the case for instance for $\fin$, $\fin\otimes\{\emptyset\}$ or $\{\emptyset\}\otimes\fin$). The same holds for tall ideals -- $\mathcal{S}$ is a tall ideal with $\bb(\mathcal{S},\mathcal{S},\mathcal{S})=\aleph_1$ whereas the ideal $\I_d$ is a tall ideal with $\bb(\I_d,\I_d,\I_d)=\bb$.
\end{remark}

\begin{remark}
$\bf{\Sigma^0_{2}}$ ideals may have different values of $\bb(\I,\I,\I)$. Indeed, there are $\bf{\Sigma^0_{2}}$ ideals with $\bb(\I,\I,\I)=\aleph_1$ (e.g. the branching ideal $\I_b$ or the Solecki's ideal $\mathcal{S}$) as well as $\bf{\Sigma^0_{2}}$ ideals with $\bb(\I,\I,\I)=\bb$ (e.g. ideals $\fin$, $\fin\otimes\{\emptyset\}$ or $\{\emptyset\}\otimes\fin$).
\end{remark}

\begin{remark}
Note that actually we cannot compute $\bb(\I,\I,\I)$ basing on the descriptive complexity of $\I$. Indeed, for each $n\geq 1$ the ideal $\fin^n$ is an $\bf{\Sigma^0_{2n}}$-complete ideal such that $\bb(\fin^n,\fin^n,\fin^n)=\bb$ (see Corollary \ref{cor:finXfinXfinXdots}). On the other hand, Calbrix in \cite{Calbrix} proved that given a $\bf{\Sigma^0_\alpha}$-complete, $\alpha>1$, (respectively: $\bf{\Pi^0_\alpha}$-complete, $\alpha>2$) subset $A$ of $2^\omega$, the ideal $\mathcal{I}_{A}$ on $2^{<\omega}$ generated by the family $\{B_x: x\in A\}$ is $\bf{\Sigma^0_\alpha}$-complete (respectively: $\bf{\Pi^0_\alpha}$-complete). Moreover, since the family $\{B_x: x\in A\}$ is almost disjoint, by Theorem \ref{AD} we get $\bb(\I_A,\I_A,\I_A)=\aleph_1$.
\end{remark}

We end with two natural open questions about $\bb(\I,\I,\I)$.

\begin{question}
Is $\bb(\I,\I,\I)\leq \dd$ for every ideal $\I$?
\end{question} 

\begin{question}
	Is it consistent that $\aleph_1<\bb(\I,\I,\I)<\bb$ for some ideal $\I$?
\end{question}


\bibliographystyle{amsplain}
\bibliography{bnumberFK}

\end{document}